\documentclass[12pt,]{article}

\usepackage[margin=1in]{geometry}
\usepackage{amscd}
\usepackage{amssymb}
\usepackage{amsmath}
\usepackage{amsthm}
\usepackage{enumerate}
\usepackage{tikz}
\usetikzlibrary{external}
\usepackage{pgfplots}
\usepackage{wrapfig, framed, caption}
\usepackage{float}
\usetikzlibrary{arrows}
\usepackage[font=small,labelfont=bf]{caption}
\usepackage{xcolor}
\usepackage{mathtools}
\usepackage[colorlinks=true, linkcolor=blue, citecolor=blue,
pagebackref=true]{hyperref}
\usepackage{fouriernc}
\usepackage{ulem}
\usepackage{comment}

\usepackage[capitalise]{cleveref}

\makeatletter
\newcommand{\refcheckize}[1]{%
  \expandafter\let\csname @@\string#1\endcsname#1%
  \expandafter\DeclareRobustCommand\csname relax\string#1\endcsname[1]{%
    \csname @@\string#1\endcsname{##1}\@for\@temp:=##1\do{\wrtusdrf{\@temp}\wrtusdrf{{\@temp}}}}%
  \expandafter\let\expandafter#1\csname relax\string#1\endcsname
}

\newcommand\blfootnote[1]{%
	\begingroup
	\renewcommand\thefootnote{}\footnote{#1}%
	\addtocounter{footnote}{-1}%
	\endgroup
}

\providecommand{\keywords}[1]
{
	\small	
	\blfootnote{\textbf{\textit{Keywords:}} #1}
}

\providecommand{\subjclass}[1]
{
	\small	
	\blfootnote{\textbf{\textit{2020 Mathematics Subject Classification:}} #1}
}

\newtheorem{theorem}{Theorem}[section]
\newtheorem*{theorem*}{Theorem}

\newtheorem*{theoremY*}{Theorem Y}

\newtheorem*{theoremAB*}{Theorem AB}

\newtheorem{corollary}{Corollary}[section]
\newtheorem*{corollary*}{Corollary}
\newtheorem{proposition}{Proposition}[section]
\newtheorem{lemma}{Lemma}[section]

\newtheorem*{claim*}{Claim}

\theoremstyle{definition}

\theoremstyle{remark}

\newtheorem*{remark*}{Remark}
\newtheorem{example}{Example}[section]

\renewcommand{\Bbb}[1]{\mathbb{#1}}

\newcommand{\bbC}{{\Bbb C}}         

\newcommand{\bbN}{{\Bbb N}}         

\newcommand{\bbR}{{\Bbb R}}        
\newcommand{\bbRP}{{\Bbb RP}} 		

\newcommand{\cA}{{\cal A}}
\newcommand{\cB}{{\cal B}}
\newcommand{\cC}{{\cal C}}

\newcommand{\cG}{{\cal G}}
\newcommand{\cH}{{\cal H}}
\newcommand{\cI}{{\cal I}}

\newcommand{\cL}{{\cal L}}

\newcommand{\cU}{{\cal U}}
\newcommand{\cV}{{\cal V}}

\renewcommand{\ge}{\geq}

\newcommand{\bi}{{\overline {\imath}}}
\newcommand{\bj}{{\overline  {\jmath}}}

\newcommand{\ai}{{\overleftarrow {\imath}}}
\newcommand{\aj}{{\overleftarrow  {\jmath}}}

\DeclareMathOperator{\dimh}{\dim_H}

\DeclareMathOperator{\proj}{proj}

\allowdisplaybreaks

\title{Hausdorff measure of dominated planar self-affine sets with large dimension}

\author{Bal\'azs B\'ar\'any\footnote{BB acknowledges support from grant NKFI~K142169, and grant NKFI KKP144059 "Fractal geometry and applications" Research Group.} \\ Department of Stochastics \\ HUN-REN–BME Stochastics Research Group \\ Institute of Mathematics \\  Budapest University of Technology and Economics \\ M\H{u}egyetem rkp. 3., H-
	1111 Budapest, Hungary  \\ email: barany.balazs@ttk.bme.hu}
\date{\today}
\begin{document}

\frenchspacing
\maketitle

\begin{abstract}
In this paper, we investigate the Hausdorff measure of planar dominated self-affine sets at its affinity dimension. We show that the Hausdorff measure being positive and finite is equivalent to the K\"aenm\"aki measure being a mass distribution. Moreover, under the open bounded neighbourhood condition, we will show that the positivity of the Hausdorff measure is equivalent to the projection of the K\"aenm\"aki measure in every Furstenberg direction being absolutely continuous with bounded density. This also implies that the affinity and the Assouad dimension coincide. We will also provide  examples for both of the cases when the Hausdorff measure is zero and positive.
\end{abstract}



%
%

\section{Introduction} \label{sec:intro}

Let $\cA$ be a finite set of indices, and let $\Phi=\{f_i(x)=A_ix+t_i\}_{i\in\cA}$ be a planar iterated function system (IFS) of affinities on $\bbR^d$ such that $\|A_i\|<1$ for every $i\in\cA$ and $|\det(A_i)|\neq 0$. Hutchinson \cite{Hutchinson1981} showed that there exists a unique non-empty compact set $X$ invariant with respect to $\Phi$, i.e.
$$
X=\bigcup_{i\in\cA}f_i(X).
$$
We call $X$ \textit{self-affine set}, and if the maps are similarities, that is, $A_i=\lambda_iO_i$, where $\lambda_i\in(0,1)$ and $O_i\in O(\bbR,d)$, then we call $X$ \textit{self-similar}. Throughout the paper, we will restrict our attention to the planar, $d=2$ case.

In the last decades, considerable attention has been paid to the geometric properties of such fractal sets, especially, to the Hausdorff dimension and measure. Let us define the Hausdorff measure, content and dimension for later purposes. For $\delta>0$ and $s\geq0$, set
$$
\cH^s_\delta(A)=\inf\{\sum_i|U_i|^s:A\subseteq\bigcup_iU_i\ \&\ |U_i|<\delta\}
$$
the $\delta$-approximation of the Hausdorff measure, { where $|U|$ denotes the diameter of the set $U$.} In particular, when $\delta=\infty$, we call the quantity $\cH^s_\infty(A)=\inf\{\sum_i|U_i|^s:A\subseteq\bigcup_iU_i\}$ the Hausdorff content. Let
$$
\cH^s(A)=\lim_{\delta\to0}\cH^s_\delta(A)\text{ and }\dimh A=\inf\{s>0:\cH^s(A)=0\}=\inf\{s>0:\cH^s_\infty(A)=0\}.
$$
be the Hausdorff measure and dimension. For basic properties, we direct the reader to the book of Falconer \cite{Falconer1990}.

Hutchinson \cite{Hutchinson1981} studied the Hausdorff dimension and measure of self-similar sets. More precisely, he showed that $\dimh(X)\leq s_0$, where $s_0$ is called the similarity dimension and it is the unique solution of the equation $\sum_{i\in\cA}\lambda_i^{s_0}=1$. Furthermore, if the IFS $\{f_i(x)=\lambda_iO_ix+t_i\}_{i\in\cA}$ satisfies the open set condition (OSC) then $0<\cH^{s_0}(X)<\infty$ and, in particular, $\dimh X=s_0$. For a precise definition of the OSC, see \cite{Hutchinson1981}. Later Bandt, Graf \cite{BandtGraf1992} and Schief \cite{Schief1994} showed that $0<\cH^{s_0}(X)<\infty$ is equivalent to the open set condition, and they gave several further equivalent characterisations.

Even if the OSC fails, and so, the $s_0$-dimensional Hausdorff measure is zero, typically the Hausdorff dimension does not drop with respect to the similarity dimension. Hochman \cite{Hochman2014} showed that if the IFS of similarities on the line satisfies the exponential separation condition then $\dimh X=\min\{1,s_0\}$. Later, Hochman \cite{hochman2017selfsimilarsetsoverlapsinverse} extended this result for higher dimensions.

Our knowledge on the more general self-affine situation is considerably more restrictive. Falconer \cite{Falconer1988} generalised the concept of the similarity dimension to the affine regime. For a $d\times d$ matrix $A$, denote $\alpha_i(A)$ the $i$th singular value of $A$. For $s\geq0$, let us define the singular value function as
$$
\varphi^s(A)=\begin{cases}
	\alpha_1(A)\cdots\alpha_{\lfloor s\rfloor}(A)\alpha_{\lceil s\rceil}(A)^{s-\lfloor s \rfloor} & \text{ if }0\leq s\leq d\\
	(|\det(A)|)^{s/d} & \text{ if }s>d.
\end{cases}
$$
We define the \textit{affinity dimension} of the IFS $\Phi=\{f_i(x)=A_ix+t_i\}_{i\in\cA}$ by
$$
s_0=\inf\left\{s>0:\sum_{n=1}^\infty\sum_{i_1,\ldots,i_n\in\cA}\varphi^s(A_{i_1}\cdots A_{i_n})<\infty\right\}.
$$
Falconer \cite{Falconer1988} showed that $s_0$ is always an upper bound for the dimension of the attractor and Solomyak \cite{Solomyak1998} proved that if $\|A_i\|<1/2$ then $\min\{d,s_0\}$ equals to the dimension for Lebesgue typical choice of translation parameters.

Unlike the self-similar case, the dimension of the attractor might drop with respect to the affinity dimension even if there is some kind of separation between the cylinder sets, like OSC. Bedford \cite{Bedford1984} and McMullen \cite{McMullen1984} studied certain type of self-affine carpets, which were later generalised by Gatzouras and Lalley \cite{GatzourasLalley} and Bara\'nski \cite{baranski}, where the matrices $A_i$ where diagonal and the set had a certain alignment structure. They gave a formula for the box-counting and Hausdorff dimension of the attractor, which is strictly smaller than the affinity dimension in most of the cases.

A possible reason for the dimension drop is the alignment structure of the set. One can get rid with it even if the matrices are diagonal by ensuring that the projections satisfies the exponential separation, {see the recent result by Rapaport \cite{Rapaportdiagonal26}.} Another way to prevent the alignment structure by assuming that the matrices satisfy the strong irreducibility assumption, namely, there is no finite collection of proper subspaces preserved by the collection of matrices. B\'ar\'any, Hochman and Rapaport \cite{BaranyHochmanRapaport2019} verified for planar systems that if the strong open set condition holds and the matrices are strongly irreducible then the Hausdorff and box-counting dimension equal to the affinity dimension. Later, Hochman and Rapaport \cite{HochmanRapaport2021} extended this for planar { exponentially} separated systems, and Rapaport \cite{rapaport2024} recently extended it for systems on $\bbR^3$ with strong open set condition.

Bedford \cite{Bedford1984}, McMullen \cite{McMullen1984} and Gatzouras and Lalley \cite{GatzourasLalley} also showed that the proper dimensional Hausdorff measure of their carpet constructions is positive and finite if and only if the box and Hausdorff dimension coincide. Peres \cite{Peres1994} studied the Hausdorff measure of Bedford-McMullen type sets in the complementary case and he showed that if the box and Hausdorff dimension do not coincide then the proper dimensional Hausdorff measure is infinite. This phenomenon has been recently extended to general Bara\'nski carpets by Qiu and Wang \cite{qiu2024hausdorff}. Kempton \cite{Kempton16} studied the slices of the so-called Przytycki-Urba\'nski carpets defined in \cite{PrzUrb}. He showed that Lebesgue almost every slice has positive $s_0-1$-dimensional Hausdorff measure (where $s_0$ is the affinity dimension of the carpet) if and only if the projection of the natural measure is absolutely continuous with bounded density. This implies that the $s_0$-dimensional Hausdorff measure is positive. This result was extended by Peng and Kamae \cite{PengKamae15} generalised for certain "function type" self-affine sets.

The aforementioned studies on the Hausdorff measure were restricted to the case when the set were carpet like, that is, there is some alignment structure. We have only a very restrictive knowledge on the Hausdorff measure in the strongly irreducible case. A direct corollary of the result of K\"aenm\"aki \cite{Kaenmaki2004} is that the $s_0$-dimensional Hausdorff measure of every self-affine set is finite, where $s_0$ is the affinity dimension. According to our best knowledge, the first result on the question under which circumstances is the $s_0$-dimensional Hausdorff measure positive in the strongly irreducible regime was due to B\'ar\'any, K\"aenm\"aki and Yu \cite{BaranyKaenmakiYu2026}. They studied dominated systems with affinity dimension smaller than $1$ and they introduced the projective separation condition which is equivalent to the positivity of the Hausdorff measure.

The goal of this paper is to extend the result of B\'ar\'any, K\"aenm\"aki and Yu \cite{BaranyKaenmakiYu2026} for the case when the affinity dimension is between $1$ and $2$.

\subsection{Main results} \label{sec:background}

Before we state the main results of the paper, let us introduce some basic notations standard in the theory of self-affine IFSs. Let $\cA$ be a finite set of indices and let us denote the usual symbolic space by $\Sigma=\cA^\bbN$, and the set of finite words by $\Sigma_*=\bigcup_{n=0}^\infty\cA^n$. { For a finite word $\bi=(i_1,\ldots,i_n)\in\Sigma_*$, denote $|\bi|$ the length of $\bi$ and denote $\bi_-$ the finite word removing the last symbol of $\bi$, that is, for $\bi=(i_1,\ldots,i_n)$, $\bi_-=(i_1,\ldots,i_{n-1})$. In case of infinite words $\bi\in\Sigma$, we use the convention $|\bi|=\infty$. Furthermore, for any $\bi=(i_1,i_2,\ldots)\in\Sigma\cup\Sigma_*$ and $n\leq|\bi|$, let $\bi|_n=(i_1,\ldots,i_n)$. We use the convention that $\bi|_0=\emptyset$. For a finite word $\bi=(i_1,\ldots,i_n)\in\Sigma_*$, let $\ai:=(i_n,\ldots,i_1)$ be the word formed by the symbols of $\bi$ in reversed order. Similarly, let us write $\ai|_n:=(i_n,\ldots,i_1)$ for $\bi=(i_1,i_2,\ldots)\in\Sigma$ and $n\geq1$.
	
For $\bi=(i_1,\ldots),\bj=(j_1,\ldots)\in\Sigma$, let $|\bi\wedge\bj|:=\min\{k\geq0:i_{k+1}\neq j_{k+1}\}$ and $\bi\wedge\bj:=\bi|_{|\bi\wedge\bj|}$. For a word $\bi\in\Sigma_*$, let $[\bi]:=\{\bj\in\Sigma:\bj|_{|\bi|}=\bi\}$ be the cylinder set, that is, all the infinite words with prefix $\bi$. We equip $\Sigma$ with the topology generated by the cylinders. This topology is coincides with the topology generated by the ultra-metric $d(\bi,\bj)=\gamma^{-|\bi\wedge\bj|}$, where $0<\gamma<1$, and $\Sigma$ is compact. The Borel sigma-algebra on $\Sigma$ coincides with the sigma-algebra generated by the cylinder sets. Finally, denote $\sigma\colon\Sigma\to\Sigma$ the usual left-shift operator.}

{ For a finite word $\bi=(i_1,\ldots,i_n)\in\Sigma_*$, let $f_\bi=f_{i_1}\circ\cdots\circ f_{i_n}$, $A_{\bi}=A_{i_1}\cdots A_{i_n}$ and $A_{\bi}^*=A_{i_1}^*\cdots A_{i_n}^*$, where $A^*$ denotes the transpose of the matrix $A$. Note that the matrix $A_{\bi}^*$ differs from the transpose of the matrix $A_{\bi}$, which would be $A_{i_n}^*\cdots A_{i_1}^*$, so whenever we require to use the transpose of the matrix $A_{\bi}$, we will use the notation $(A_{\bi})^*$ or $A_{\ai}^*$. We use the same convention for the inverses of matrix products.} Let us define the natural projection $\pi\colon\Sigma\to X$ by
\begin{equation}\label{eq:natproj}
	\pi(\bi):=\lim_{n\to\infty}f_{\bi|_n}(0).
\end{equation} { Clearly, $\pi(\bi)=f_{i_1}(\pi(\sigma\bi))$, and it is H\"older continuous with respect to the metric $d$ on $\Sigma$ and the usual Euclidean metric.}

Throughout the paper, we will assume that the collection of matrices $\{A_i\}_{i\in\cA}$ is \textit{dominated}. That is, there exist $C>0$ and $0<\tau<1$ such that
\begin{equation}\label{eq:domindef}
\alpha_2(A_{\bi})\leq C\tau^{|\bi|}\alpha_1(A_{\bi})\text{ for every $\bi\in\Sigma_*$.}
\end{equation}
Bochi and Gourmelon \cite[Theorem~A]{BochiGourmelon2009} showed that the matrices $\{A_i\}_{i\in\cA}$ are dominated if and only if $\{A_i\}_{i\in\cA}$ admits a strongly invariant multicone. We say that a proper subset $\cC\subset\bbRP^1$ is a \textit{multicone} if it is a finite union of closed projective intervals. Moreover, we say that a multicone $\cC$ is strongly invariant if $\bigcup_{i\in\cA}A_i^*\cC\subseteq\cC^o$. Let us define the collection of \textit{Furstenberg directions} by $X_F=\bigcap_{n=0}^\infty\bigcup_{\bi:|\bi|=n}A_{\bi}^*\cC$. We define, similarly to the natural projection, a map $V\colon\Sigma\to X_F$ by
\begin{equation}\label{eq:furstproj}
	\{V(\bi)\}=\bigcap_{n=1}^\infty A_{\bi|_n}^*\cC.
\end{equation}
One can easily see that $V(\bi)=A_{i_1}^*V(\sigma\bi)$, {and $V$ is H\"older continuous with respect to the metric $d$ on $\Sigma$ and the metric $\rho$ on the projective space $\bbRP^1$, where
$$
\rho(V,W)=|\sin(\sphericalangle(V,W))|\text{ for }V,W\in\bbRP^1,
$$
and $\sphericalangle$ denotes the angle between the subspaces $V$ and $W$.} With a slight abuse of the notation, we will say that the IFS $\Phi=\{f_i(x)=A_ix+t_i\}_{i\in\cA}$ is dominated if the set of linear parts $\{A_i\}_{i\in\cA}$ is dominated.

If $\{A_i\}_{i\in\cA}$ is dominated then there exists a unique left-shift invariant ergodic probability measure $\mu_K$ on $\Sigma$ such that there exists $c>0$ such that
\begin{equation}\label{eq:Kaenmakimeas}
	c^{-1}\varphi^{s_0}(A_{\bi})\leq\mu_K([\bi])\leq c\varphi^{s_0}(A_{\bi}),
\end{equation}
see K\"aenm\"aki \cite{Kaenmaki2004} and B\'ar\'any, K\"aenm\"aki and Morris \cite{BaranyKaenmakiMorris2018}. A simple combination of the existence of the K\"aenm\"aki measure \cref{eq:Kaenmakimeas} and the covering argument by Falconer \cite{Falconer1988} implies that $\cH^{s_0}(X)<\infty$. For a proof, see \cite[Lemma~2.18]{BaranyKaenmakiYu2026}.

For $V\in\bbRP^1$, denote $\proj_V\colon\bbR^2\to V$ the orthogonal projection onto the subspace $V$, and let us denote by $\lambda_V$ the Lebesgue measure on $V$. Now, we are ready to state our main result.

\begin{theorem}\label{thm:main}
	Let $\Phi=\{f_i(x)=A_ix+t_i\}_{i\in\cA}$ be a dominated planar IFS of affinities  with affinity dimension $s_0\in(1,2]$. Let $X$ be the attractor of $\Phi$, let $\mu_K$ be the K\"aenm\"aki measure and let $\pi$ be the natural projection. Then the following are equivalent:
	\begin{enumerate}[(a)]
		\item\label{it:posmeas} $\cH^{s_0}(X)>0$;
		\item\label{it:posintexist} there exists $V\in X_F$ such that $\int \cH_\infty^{s_0-1}(X\cap\proj_V^{-1}(t))d\lambda_V(t)>0$;
		\item\label{it:posintevery} $\inf_{V\in X_F}\int \cH_\infty^{s_0-1}(X\cap\proj_V^{-1}(t))d\lambda_V(t)>0$;
		\item\label{it:massdist} there exists a constant $C>0$ such that $\pi_*\mu_K(B(x,r))\leq C\cdot r^{s_0}$ for every $x\in X$, $r>0$, where $B(x,r)$ denotes the ball with radius $r$ centred at $x$.
	\end{enumerate}
\end{theorem}

Unlike to the self-similar case, see Bandt and Graf \cite{BandtGraf1992} and Schief \cite{Schief1994}, and unlike to the dominated self-affine case with $s_0\leq1$, see B\'ar\'any, K\"aenm\"aki and Yu \cite{BaranyKaenmakiYu2026}, $\cH^{s_0}(X)>0$ does not imply the $s_0$-Ahflors regularity of $X$. In particular, B\'ar\'any, K\"aenm\"aki and Yu \cite{BaranyKaenmakiYu2026} showed that for a dominated planar self-affine set with strong separation if $s_0>1$ then $X$ cannot be $s_0$-Ahlfors regular. However, \cref{thm:main} shows that the positivity of the $s_0$-dimension Hausdorff measure is equivalent to a very rigid geometric structure, which is not easy-to-verify.

The positivity of the Hausdorff measure has some further consequences:

\begin{theorem}\label{thm:maincor}
	Let $\Phi=\{f_i(x)=A_ix+t_i\}_{i\in\cA}$ be a dominated planar IFS of affinities  with affinity dimension $s_0\in(1,2]$. Let $X$ be the attractor of $\Phi$, let $X_F$ be the set of Furstenberg directions, let $\mu_K$ be the K\"aenm\"aki measure and let $\pi$ be the natural projection. If $\cH^{s_0}(X)>0$ then
\begin{enumerate}[(i)]
	\item\label{it:boundeddens} there exists a constant $C>0$ such that $(\proj_V)_*\pi_*\mu_K(B(t,r))\leq C\cdot r$ for every $V\in X_F$, $t\in \proj_V(X)$ and $r>0$;
	\item\label{it:boundedmeas} there exists $C>0$ such that for every $V\in X_F$ and for every $t\in \proj_V(X)$,\linebreak $\cH^{s_0-1}(X\cap\proj_V^{-1}(t))\leq C$.
\end{enumerate}
\end{theorem}

Clearly, \eqref{it:boundedmeas} cannot be equivalent to $\cH^{s_0}(X)>0$. For example, if the maps of $\Phi$ have a common fixed point, that is, $X$ is a singleton, but $s_0\in(1,2]$ then $\cH^{s_0-1}(X\cap\proj_V^{-1}(t))=0\leq C$ for every $t\in\proj_V(X)$, however, $\cH^{s_0}(X)=0$. Item \eqref{it:boundeddens} seems strong enough to be equivalent to $\cH^{s_0}(X)>0$ in the generality of \cref{thm:maincor}, but we could not verify it. For this reason, we introduce the open bounded neighbourhood condition motivated by the bounded neighbourhood condition introduced by Anttila, B\'ar\'any, K\"aenm\"aki \cite{AnttilaBaranyKaenmaki24}. For $r>0$, let
\begin{equation}\label{eq:Delta}
\Delta_r=\{\bi\in\Sigma_*:\alpha_2(A_{\bi})|X|\leq r<\alpha_2(A_{\bi_-})|X|\}.
\end{equation}
We say that $\Phi$ satisfies the \textit{open bounded neighbourhood condition (OBNC)} if there exists an open and bounded set $U$ such that $f_i(U)\subseteq U$ by every $i\in\cA$ and there exists $C>0$ such that for every $r>0$ and every $x\in\bbR^2$
$$
\#\{\bi\in\Delta_r:f_\bi(U)\cap B(x,r)\neq\emptyset\}\leq C.
$$
It is easy to see that the strong separation condition implies the OBNC, and the OBNC implies the bounded neighbourhood condition defined in \cite[Section~2.5]{AnttilaBaranyKaenmaki24}, but the strong open set condition does not imply bounded neighbourhood condition, see \cite[Example~3.3]{AnttilaBaranyKaenmaki24}.

\begin{theorem}\label{thm:mainsep}
		Let $\Phi=\{f_i(x)=A_ix+t_i\}_{i\in\cA}$ be a dominated planar IFS of affinities  with affinity dimension $s_0\in(1,2]$. Let $X$ be the attractor of $\Phi$, let $X_F$ be the set of Furstenberg directions, let $\mu_K$ be the K\"aenm\"aki measure and let $\pi$ be the natural projection. Furthermore, suppose that $\Phi$ satisfies the open bounded neighbourhood condition. Then the following are equivalent:
	\begin{enumerate}[(1)]
		\item\label{it:posmeas2} $\cH^{s_0}(X)>0$;
		\item\label{it:boundeddens1} there exists a constant $C>0$ such that $(\proj_V)_*\pi_*\mu_K(B(t,r))\leq C\cdot r$ for every $V\in X_F$, $t\in \proj_V(X)$ and $r>0$.
	\end{enumerate}
\end{theorem}

We note that \eqref{it:boundedmeas} in \cref{thm:maincor} (same as \eqref{it:boundeddens1} in \cref{thm:mainsep}) has already appeared as a sufficient condition in the recent paper of Batsis, K\"aenm\"aki and Kempton \cite[Theorem~1.3]{batsis2026} regarding the multifractal analysis of fully supported quasi-Bernoulli measures on dominated planar self-affine sets. One might wonder whether is it enough to verify the bounded density of $(\proj_V)_*\pi_*\mu_K$ for only one $V\in X_F$. It seems very likely but we were unable to prove it.

A corollary of \cref{thm:maincor} and the estimate of Anttila, B\'ar\'any and K\"aenm\"aki \cite[Proposition~3.1]{AnttilaBaranyKaenmaki24} is the following:

\begin{corollary}\label{cor}
	Let $\Phi=\{f_i(x)=A_ix+t_i\}_{i\in\cA}$ be a planar IFS of affinities with affinity dimension $s_0\in(1,2]$. Suppose that $\Phi$ is dominated and satisfies the open bounded neighbourhood condition. Denote $X$ the attractor of $\Phi$. If $\cH^{s_0}(X)>0$ then $\dim_{\rm A}X=s_0$, where $\dim_A$ denotes the Assouad dimension of $X$.
\end{corollary}

For precise definition and properties of the Assouad dimension, see Fraser \cite{Fraser2020}.

\subsection{Examples}\label{subsec:examples}

Finally, we consider some examples for our main theorems. First, we consider a strongly irreducible example with attractor having zero proper dimensional Hausdorff measure. This example has already appeared in \cite[Example~3.3]{BaranyKaenmakiYu2026}.

\begin{example}
	 Let $q > p \ge 2$ and $p<N \in \{2,\ldots,pq\}$ be integers, and let $I \subset \{0,\ldots,p-1\} \times \{0,\ldots,q-1\}$ be a set of $N$ elements. Let $A = \begin{pmatrix}\tfrac{1}{p} & 0\\ 0 &\tfrac{1}{q}\end{pmatrix}$ and let $B$ be a $2\times 2$ matrix with $\det(B)>0$ and with strictly positive entries such that $\|B\|<1$. It is easy to see that the matrices $\{A,B\}$ are dominated and strongly irreducible with Furstenberg directions containing the $x$-axis.
	
	 Let $\epsilon>0$ and $t\in\bbR^2$ be such that the IFS
	 $$
	 \Phi_\epsilon=\left\{x\mapsto Ax+\begin{pmatrix}j/p\\k/q\end{pmatrix}\right\}_{(j,k) \in I}\cup\{x\mapsto \epsilon\cdot Bx+t\}
	 $$
	 satisfies that $f([0,1]^2)\cap g([0,1]^2)=\emptyset$ for every $f\neq g\in\Phi$, and
	 $$
	 \max_{j \in \{1,\ldots,p\}} \#\{i:(j,i)\in I\}>\frac{N}{p}>1.
	 $$
	 Let $j'$ be the symbol for which the maximum on the left-hand side is attained. By B\'ar\'any, Hochman and Rapaport \cite{BaranyHochmanRapaport2019}, $\dimh X=s_0(\epsilon)$, where $s_0(\epsilon)$ is the affinity dimension and $X$ is the attractor of $\Phi_\epsilon$. For the images of the first level cylinder sets, see \cref{fig:picture1}.
	
	 Since $\varphi^s(\epsilon\cdot B) \to 0$ as $\epsilon \to 0$ for all $s \ge 0$, the affinity dimension $s_0$ of $\Phi_\epsilon$ converges to
	 $1+\frac{\log N-\log p}{\log q}$ as $\epsilon\to0$. Hence, one can choose $\epsilon>0$ sufficiently small such that
	 $$
	 s_0(\epsilon)-1<\frac{\log\#\{i:(j',i)\in I\}}{\log q}.
	 $$
	
	 Since the $x$-axis belongs to $X_F$, and the attractor of the IFS $\left\{Ax+\left(\tfrac{j'}{p},\tfrac{k}{q}\right)\right\}_{(j',k) \in I}$ forms a slice of $X$ with dimension $\frac{\log\#\{i:(j',i)\in I\}}{\log q}>s_0(\epsilon)-1$, by \cref{thm:maincor}\eqref{it:boundedmeas}, we have $\cH^{s_0(\epsilon)}(X)=0$.
\end{example}

\begin{figure}
	\centering
	\includegraphics[width=0.33\linewidth]{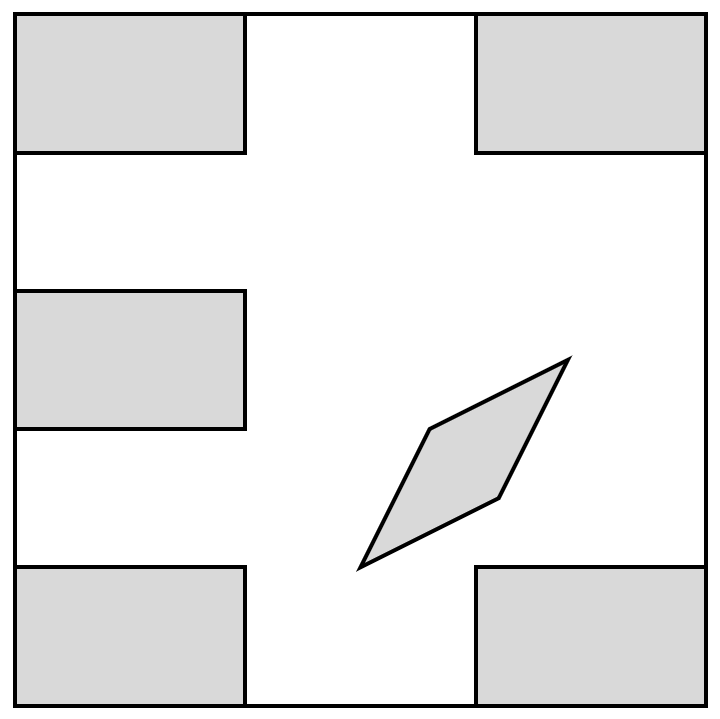}
	\caption{First level cylinder sets of the IFS $\Phi=\{f_1(x,y)=\left(\frac{x}{3},\frac{y}{5}\right),f_2(x,y)=\left(\frac{x}{3},\frac{y+2}{5}\right),f_3(x,y)=\left(\frac{x}{3},\frac{y+4}{5}\right),f_4(x,y)=\left(\frac{x+2}{3},\frac{y}{5}\right),f_5(x,y)=\left(\frac{x+2}{3},\frac{y+4}{5}\right),f_6(x,y)=\left(\frac{2x+y+5}{10},\frac{x+2y+2}{10}\right)\}$. Simple calculation shows that the Hausdorff dimension is at most $1.607$ but the largest horizontal slice has dimension $0.6826$, and so, the proper dimensional Hausdorff measure is zero.}
	\label{fig:picture1}
\end{figure}

Now, we provide two triangular examples with positive and finite $s_0$-dimensional Hausdorff measure. Unfortunately, our examples are not strongly irreducible, the linear parts of the maps of the IFS are lower triangular matrices. However, we provide examples for both cases when $X_F$ is and is not a singleton. First, we consider an example when $X_F$ is a singleton.

\begin{example}\label{ex:singlexf}
Let $\cA$ be a finite set of indices and for every $i\in\cA$, let $0<|a_i|<|c_i|<1$	such that $\max_i|c_i|<1/2$, $\sum_{i\in\cA}|c_i||a_i|^{1/4}>1$ and $\sum_{i\in\cA}|a_i|^{1/2}<1$. Let
$$
\Phi=\left\{f_i(x)=\begin{pmatrix}
	a_i & 0 \\ 0 & c_i
\end{pmatrix}x+t_i\right\}_{i\in\cA}.
$$
and denote $X$ the attractor of $\Phi$. Then $0<\cH^{s_0}(X)<\infty$ for Lebesgue-almost every $(t_{i})_{i\in\cA}\in\bbR^{2\#\cA}$, where $\sum_{i\in\cA}|c_i||a_i|^{s_0-1}=1$.
\end{example}

For example, the choices $\#\cA=10$, $c_i=\tfrac13$ and $a_i=\tfrac{1}{121}$ satisfies the assumptions of \cref{ex:singlexf}.

Now, let us consider an example with positive Hausdorff measure for which $\dimh X_F>0$.

\begin{example}\label{ex:posdimxf}
	Let
\begin{equation}\label{eq:ex2}
	\Phi=\left\{f_i(x)=\begin{pmatrix}
		a_i & 0 \\ b_i & c_i
	\end{pmatrix}x+\begin{pmatrix}t_{i,1}\\t_{i,2}\end{pmatrix}\right\}_{i\in\cA}
\end{equation}
be an IFS such that $0<|a_i|<|c_i|<1/2$, $\sum_{i\in\cA}|c_i|>1$ and  the IFS $\Phi_1=\{x\to a_ix+t_{i,1}\}_{i\in\cA}$ satisfies the strong open set condition. Denote $s_0$ the affinity dimension $\sum_{i\in\cA}|c_i||a_i|^{s_0-1}=1$, $s_0\in(1,2]$. If $\sum_{i\in\cA}|c_i|^{-1}|a_i|^{2(s_0-1)}<1$ then { $0<\cH^{s_0}(X)<\infty$} for Lebesgue-almost every $\tau=(t_{i,2})_{i\in\cA}$, where $X$ is the attractor of $\Phi$.
\end{example}

For $N\geq28$, the choices $\cA=\{0,\ldots,N-1\}$ and $a_i=\frac{1}{N+1}$, $t_{i,1}=\frac{i\cdot N}{N^2-1}$, $c_i=\frac13$ for every $i\in\cA$ satisfy the assumption of \cref{ex:posdimxf}. For a visualisation of the examples, see \cref{fig:picture2}.

\begin{figure}
	\centering
	\includegraphics[width=0.33\linewidth]{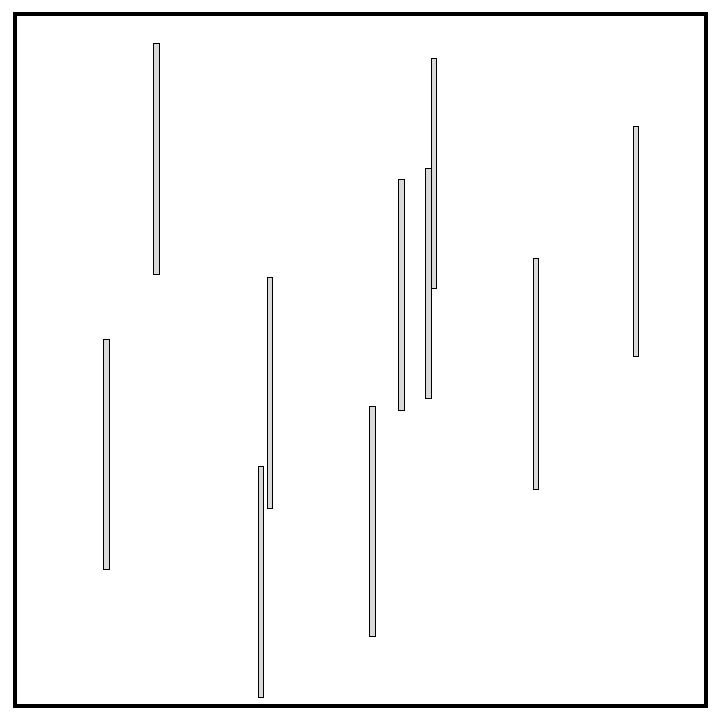}\includegraphics[width=0.33\linewidth]{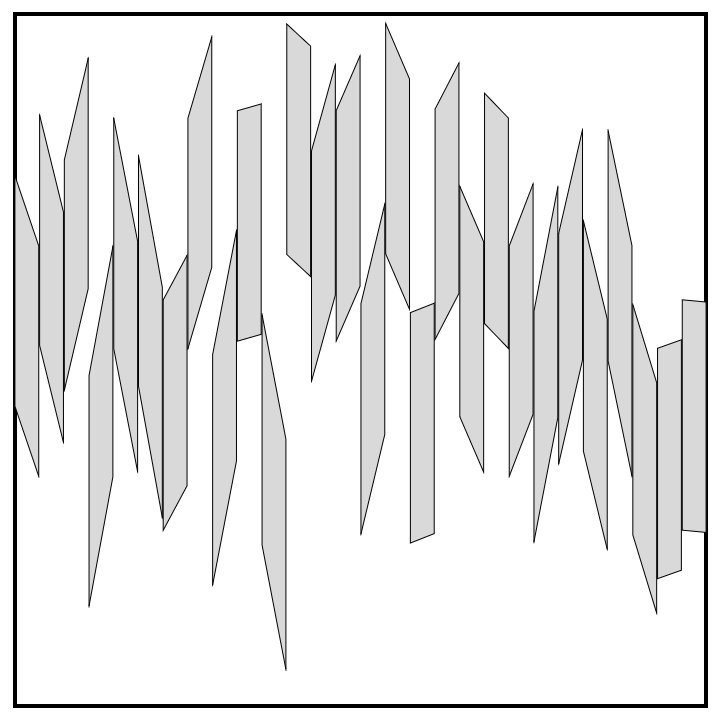}
	\caption{First level cylinder sets of the IFSs in \cref{ex:singlexf} and \cref{ex:posdimxf}, which has positive and finite Hausdorff measure.}
	\label{fig:picture2}
\end{figure}

We will verify \cref{ex:singlexf} and \cref{ex:posdimxf} in \cref{sec:examples}.

It is a natural question how typical the positivity of the Hausdorff measure is. From the examples, we saw that for given linear parts the $s_0$-dimensional Hausdorff measure is positive for almost every translation parameters. Is it true in general that for a typical choice of parameters in some proper sense the $s_0$-dimensional Hausdorff measure is positive?

\section{Preliminaries}

Throughout this paper, we will always assume that $\Phi=\{f_i(x)=A_ix+t_i\}_{i\in\cA}$ is dominated, and $s_0\in(1,2]$, where $s_0$ is the affinity dimension. Without loss of generality, we will always assume that $X\subseteq B(0,1)$. From the domination by \cite[Theorem~A]{BochiGourmelon2009}, it follows that there exists a multicone $\cC\subset\bbRP^1$ such that $A_i^*\cC\subseteq\cC^o$ for every $i\in\cA$. Then it is easy to see that $A_i^{-1}\cC^\perp\subseteq(\cC^\perp)^o$, where $\cC^\perp=\{V\in\bbRP^1:V^\perp\in\cC\}$.

Let $V\colon\Sigma\to X_F$ be the natural projection to the set of Furstenberg directions defined in \cref{eq:furstproj}. It is clear that $V\colon\Sigma\to X_F$ is H\"older-continuous. Moreover,
$$
V(\bi)=A_{i_1}^*V(\sigma\bi)\text{ and }V(\bi)^\perp=A_{i_1}^{-1}V(\sigma\bi)^\perp.
$$
By \cite[Lemma~2.2]{BochiMorris2015}, there exists a constant $C>1$ such that  and every $\bi\in\Sigma_*$
\begin{equation}\label{eq:domin}
	\|A_{\bi}^*|V\|\leq \alpha_1(A_{\bi}^*)=\alpha_1(A_{\ai})\leq C\|A_{\bi}^*|V\|\text{ and }\|A_{\bi}^{-1}|V^\perp\|\leq \alpha_1(A_{\bi}^{-1})=\alpha_2(A_{\ai})^{-1}\leq C\|A_{\bi}^{-1}|V^\perp\|
\end{equation}
for every $V\in\bigcup_{i\in\cA}A_i^*\cC$. { Let us note that above (and throughout the paper) we used the convention that $A_{\bi}^*=A_{i_1}^*\cdots A_{i_n}^*$ and $A_{\bi}^{-1}=A_{i_1}^{-1}\cdots A_{i_n}^{-1}$ for $\bi=(i_1,\ldots,i_n)$.}

With a slight abuse of notation, we define the orthogonal projection $\proj_V$ as real valued function over $V\in\cC$ as follows: for every $V\in\cV$, let $v=v(V)\in V$ be a unit vector such that the map $V\mapsto v$ is continuous on $\cC$, and let
$$
\proj_V(x)=\langle v(V),x\rangle,
$$
where $\langle\cdot,\cdot\rangle$ denotes the usual scalar product on $\bbR^d$. Note that $\proj_V\colon\bbR^2\to\bbR$ is bi-Lipschitz equivalent to { the actual orthogonal projection $x\mapsto \langle v(V),x\rangle v(V)$, since it identifies $V$ to $\bbR$ with some prefixed orientation.} Let us denote the Lebesgue measure on $[-1,1]$ by $\lambda$.

By defining $F_{i,V}\colon\bbR\to\bbR$ such that
$$
F_{i,V}=\begin{cases}
	\|A_i^*|V\|x+\proj_V(t_i) & \text{ if }\frac{A_i^*v(V)}{\|A_i^*v(V)\|}=v(A_i^*V)\\
	-\|A_i^*|V\|x+\proj_V(t_i) & \text{ otherwise.}
\end{cases}
$$
Simple calculation shows that $\proj_V(f_i(x))=F_{i,V}(\proj_{A_i^*V}(x))$ for every $x\in\bbR^2$.

\subsection{Perron-Frobenius operator and its eigenfunction}\label{sec:perron}

We define a H\"older-continuous potential $g\colon\Sigma\to\bbR$ as follows:
\[
g(\bi):=\log\|A_{i_1}^*|V(\sigma\bi)\|+(s_0-1)\log\|A_{i_1}|V(\bi)^\perp\|=\log\|A_{i_1}^*|V(\sigma\bi)\|-(s_0-1)\log\|A_{i_1}^{-1}|V(\sigma\bi)^\perp\|.
\]
Simple calculation shows that for every $\bi=(i_1,i_2,\ldots)\in\Sigma$ and $n\geq1$
\[
\sum_{k=0}^{n-1}g(\sigma^k\bi)=\log\|A_{i_1}^*\cdots A_{i_n}^*|V(\sigma^n\bi)\|-(s_0-1)\log\|A_{i_1}^{-1}\cdots A_{i_n}^{-1}|V(\sigma^n\bi)^\perp\|,
\]
and so
$$
\varphi^{s_0}(A_{\ai|_n})-\log C\leq\sum_{k=0}^{n-1}g(\sigma^k\bi)\leq\varphi^{s_0}(A_{\ai|_n})+\log C,
$$
where $\ai|_n=(i_n,\ldots,i_1)$ for $\bi=(i_1,i_2,\ldots)$.

{ Let us define the Perron-Frobenius operator $\cL$, which maps real valued functions over $\Sigma$ to real-valued functions over $\Sigma$ such that
$$
(\cL p)(\bi)=\sum_{k\in\cA}e^{g(k\bi)}p(k\bi)=\sum_{k\in\cA}\|A_k^*|V(\bi)\|\cdot\|A_{k}^{-1}|V(\bi)^\perp\|^{-(s_0-1)}\cdot p(k\bi).
$$
Since the map $\bi\mapsto e^{g(\bi)}$ is Hölder-continuous, $\cL$ maps continuous maps to continuous maps.} By Ruelle's Perron-Frobenius Theorem (see for example \cite[Theorem~1.7]{Bowen}), there exists a unique continuous function $p\colon\Sigma\to\bbR$ with $p(\bi)>0$ for every $\bi\in\Sigma$, and there exists a unique Borel probability measure $\nu$ for which $\cL p=p$, $\cL^*\nu=\nu$, $\int p(\bi)d\nu(\bi)=1$ and
\begin{equation}\label{eq:conv}
\lim_{n\to\infty}\sup_{\bi\in\Sigma}\left|(\cL^nh)(\bi)-p(\bi)\int h(\bi)d\nu(\bi)\right|=0\text{ for every $h\colon\Sigma\to\bbR$ continuous}.
\end{equation}
We define $\mu_F([\bi]):=\int_{[\bi]}p(\bj)d\nu(\bj)$, then $\mu_F$ is ergodic left-shift invariant probability measure such that for every $\bi\in\Sigma_*$ and $\bj\in\Sigma$
$$
C'^{-1}\varphi^{s_0}(A_{\ai})\leq C^{-1}\exp\left(\sum_{k=0}^{|\bi|-1}g(\sigma^k\bi\bj)\right)\leq\mu_F([\bi])\leq C\exp\left(\sum_{k=0}^{|\bi|-1}g(\sigma^k\bi\bj)\right)\leq C'\varphi^{s_0}(A_{\ai}).
$$
Note that $\mu_F$ is the "reversed" K\"aenm\"aki measure, that is, $\mu_F([\bi])=\mu_K([\ai])$, where $\ai=(i_n,\ldots,i_1)$ for $\bi=(i_1,\ldots,i_n)$, which follows from the uniqueness of the K\"aenm\"aki measure under domination, see \cite{BaranyKaenmakiMorris2018}.

\subsection{Hausdorff content of slices}

Now, let us define a map $h\colon\Sigma\to\bbR$ as follows
$$
h(\bi):=\int \cH_\infty^{s_0-1}\left(X\cap\proj_{V(\bi)}^{-1}(t)\right)d\lambda(t),
$$
{ where we recall that $\lambda$ is the Lebesgue measure on $[-1,1]$.} We will show that $h$ is a constant multiplier of the eigenfunction $p$ of $\mathcal{L}$. The proof is similar to the proof of \cite[Lemma~7.1]{BaranyKaenmakiYu2026}. Let us begin with some technical lemmas.
{
\begin{lemma}\label{lem:semicont0}
	The map $(V,t)\mapsto \cH_\infty^{s_0-1}\left(X\cap\proj_{V}^{-1}(t)\right)$ is upper semi-continuous.
\end{lemma}

\begin{proof}
To verify the claim let $(V_n,t_n)$ such that $(V_n,t_n)\to (V,t)$. Then for any sequence $x_n\in X\cap \proj_{V_n}^{-1}(t_n)$ such that $x_n\to x$ then $x\in X\cap \proj_{V}^{-1}(t)$ by the compactness of $X$. So, for $\varepsilon>0$ if $\{U_i\}$ is an open cover of $X\cap \proj_{V}^{-1}(t)$ such that $\sum_i |U_i|^{s_0-1}\leq\cH_\infty^{s_0-1}\left(X\cap\proj_{V}^{-1}(t)\right)+\varepsilon$ then without loss of generality, we may assume that $\{U_i\}$ is finite (by the compactness of $X\cap\proj_{V}^{-1}(t)$), and so, $X\cap \proj_{V_n}^{-1}(t_n)\subseteq\bigcup_iU_i$ for every sufficiently large $n$, which proves the claim
\end{proof}

\begin{lemma}\label{lem:semicont}
	The map $V\mapsto \int\cH_\infty^{s_0-1}\left(X\cap\proj_{V}^{-1}(t)\right)d\lambda(t)$ is upper semi-continuous. In particular, the map $\bi\mapsto h(\bi)$ is upper semi-continuous.
\end{lemma}

\begin{proof}
	  Let $V\in\bbRP^1$ be arbitrary and fixed and let $V_n$ be a sequence such that $V_n\to V$. Since $X\subset B(0,1)$, we have that for every $W\in\bbRP^1$ and $t\in[-1,1]$
	  $$
	  0\leq\cH^{s_0-1}_\infty\left(X\cap\proj_W^{-1}(t)\right)\leq2^{s_0-1}.
	  $$
	  Thus,
	  \begin{align*}
	  2^{s_0}-\int\cH^{s_0-1}_\infty\left(X\cap\proj_V^{-1}(t)\right)d\lambda(t)&=\int 2^{s_0-1}-\cH^{s_0-1}_\infty\left(X\cap\proj_V^{-1}(t)\right)\,d\lambda(t)\\
	  &\leq\int 2^{s_0-1}-\limsup_{n\to\infty}\cH^{s_0-1}_\infty\left(X\cap\proj_{V_n}^{-1}(t)\right)\,d\lambda(t)\\
	  &=\int \liminf_{n\to\infty}\left(2^{s_0-1}-\cH^{s_0-1}_\infty\left(X\cap\proj_{V_n}^{-1}(t)\right)\right)\,d\lambda(t)\\
	  &\leq\liminf_{n\to\infty}\int 2^{s_0-1}-\cH^{s_0-1}_\infty\left(X\cap\proj_{V_n}^{-1}(t)\right)d\lambda(t)\\
	  &=2^{s_0}-\limsup_{n\to\infty}\int\cH^{s_0-1}_\infty\left(X\cap\proj_{V_n}^{-1}(t)\right)d\lambda(t),
	  \end{align*}
	  where we used \cref{lem:semicont0} in the first inequality, and Fatou's lemma in the second. This proves the first assertion of the lemma, and since $\bi\mapsto V(\bi)$ is continuous, the second assertion follows too.
\end{proof}
}

\begin{lemma}\label{lem:operator}
	For every $\bi\in\Sigma$, $h(\bi)\leq(\cL h)(\bi)$.
\end{lemma}

\begin{proof}
It is easy to see that for every $V\in\bbRP^1$, $x,y\in\bbR^2$ and $i\in\cA$
$$
\|\proj_V(f_i(x))-\proj_V(f_i(y))\|=\|A_i^*|V\|\cdot\|\proj_{A_i^*V}(x-y)\|.
$$
So
\begin{align*}
	\int \cH_\infty^{s_0-1}&\left(X\cap\proj_{V(\bi)}^{-1}(t)\right)d\lambda(t)\\
	&\leq\sum_{k\in\cA}\int \cH_\infty^{s_0-1}\left(f_k(X)\cap\proj_{V(\bi)}^{-1}(t)\right)d\lambda(t)\\
	&=\sum_{k\in\cA}\int_{\proj_{V(\bi)}(f_k(X))} \cH_\infty^{s_0-1}\left(f_k(X)\cap\proj_{V(\bi)}^{-1}(t)\right)d\lambda(t)\\
	&=\sum_{k\in\cA}\int_{\proj_{V(k\bi)}(X)} \cH_\infty^{s_0-1}\left(f_k(X)\cap\proj_{V(\bi)}^{-1}(F_{k,V}(t))\right)\|A_k^*|V(\bi)\|d\lambda(t)\\
	&=\sum_{k\in\cA}\|A_k^*|V(\bi)\|\int_{\proj_{V(k\bi)}(X)} \cH_\infty^{s_0-1}\left(f_k\left(X\cap\proj_{V(k\bi)}^{-1}(t)\right)\right)d\lambda(t)\\
	&=\sum_{k\in\cA}\|A_k^*|V(\bi)\|\int_{\proj_{V(k\bi)}(X)} \cH_\infty^{s_0-1}\left(X\cap\proj_{V(k\bi)}^{-1}(t)\right)\|A_k|V(k\bi)^\perp\|^{s_0-1}d\lambda(t)\\
	&=(\cL h)(\bi).
\end{align*}
\end{proof}

\begin{proposition}\label{prop:Hauscontent}
	If $p\colon\Sigma\to(0,\infty)$ is the map and $\nu$ is the measure defined by Ruelle's Perron-Frobenius Theorem in \cref{sec:perron}, then we get
	$$
	h(\bi)=p(\bi)\iint \cH_\infty^{s_0-1}\left(X\cap\proj_{V(\bj)}^{-1}(t)\right)d\lambda(t)d\nu(\bj).
	$$
	In particular, either $h\equiv0$ or $\inf_{\bi\in\Sigma}h(\bi)>0$.
\end{proposition}

\begin{proof}
	Since $h\colon\Sigma\to\bbR$ is upper semi-continuous by \cref{lem:semicont}, for every $n\geq1$ there exists a continuous function $h_n\colon\Sigma\to\bbR$ such that $h(\bi)\leq h_n(\bi)$ and $\int h_n(\bi)d\nu(\bi)\leq \int h(\bi)d\nu(\bi)+1/n$ by \cite[Theorem~2.1.3]{Ransford} and the monotone convergence theorem. Then by \cref{eq:conv}
	\[
	h(\bi)\leq\liminf_{k\to\infty}(\cL^kh)(\bi)\leq\liminf_{k\to\infty}(\cL^kh_n)(\bi)=p(\bi)\int h_nd\nu\leq p(\bi)\left(\int hd\nu+1/n\right).
	\]
	Since $n\geq1$ was arbitrary
	$$
	h(\bi)\leq p(\bi) \int hd\nu.
	$$
	{ Moreover,
	$$
	\int p(\bi)\int hd\nu-h(\bi)\,d\nu(\bi)=\int hd\nu\,\left(\int pd\nu-1\right)=0.
	$$
	In summary, the function $\bi\mapsto p(\bi)\int hd\nu-h(\bi)$ is non-negative, lower semi-continuous and has integral zero on a fully supported measure, and hence, it is identically zero, which had to be shown.}
\end{proof}

\subsection{Hausdorff measure of slices}

\begin{proposition}\label{prop:Hausmeas} Let $h\colon\Sigma\to[0,\infty)$ be the function defined in \cref{prop:Hauscontent}. Then
	$$
	h(\bi)=\int \cH^{s_0-1}\left(X\cap\proj_{V(\bi)}^{-1}(t)\right)d\lambda(t).
	$$
\end{proposition}

\begin{proof}
	Let $n\geq1$ be such that $|f_{\bj}(X)|\leq\delta$ for every $\bj$ with $|\bj|\geq n$. Then
	$$
	\cH^{s_0-1}_\delta\left(f_{\bj}(X)\cap\proj_{V(\bi)}^{-1}(t)\right)=\cH^{s_0-1}_\infty\left(f_{\bj}(X)\cap\proj_{V(\bi)}^{-1}(t)\right).
	$$
	Thus, similarly to the proof of \cref{lem:operator}, for every $\delta>0$ and $\bi\in\Sigma$ we get
	\begin{align*}
	\int_{\proj_{V(\bi)}(X)} &\cH^{s_0-1}_\delta\left(X\cap\proj_{V(\bi)}^{-1}(t)\right)d\lambda(t)\\
	&\leq \sum_{|\bj|=n}\int_{\proj_{V(\bi)}(X)} \cH^{s_0-1}_\delta\left(f_{\bj}(X)\cap\proj_{V(\bi)}^{-1}(t)\right)d\lambda(t)\\
	&=\sum_{|\bj|=n}\int_{\proj_{V(\bi)}(X)} \cH^{s_0-1}_\infty\left(f_{\bj}(X)\cap\proj_{V(\bi)}^{-1}(t)\right)d\lambda(t)\\
	&=\sum_{|\bj|=n}\int_{\proj_{V(\bi)}(f_{\bj}(X))} \cH^{s_0-1}_\infty\left(f_{\bj}(X)\cap\proj_{V(\bi)}^{-1}(t)\right)d\lambda(t)\\
	&=\sum_{|\bj|=n}\int_{\proj_{V(\bi)}(f_{\bj}(X))} \cH^{s_0-1}_\infty\left(f_{\bj}\left(X\cap\proj_{V(\aj\bi)}^{-1}(F_{\bj,V(\bi)}(t))\right)\right)d\lambda(t)\\
	&=\sum_{|\bj|=n}\|A_{\bj}^*|V(\bi)\|\int_{\proj_{V(\aj\bi)}(X)} \cH^{s_0-1}_\infty\left(f_{\bj}\left(X\cap\proj_{V(\aj\bi)}^{-1}(t)\right)\right)d\lambda(t)\\
	&=\sum_{|\bj|=n}\|A_{\bj}^*|V(\bi)\|\|A_{\bj}|V(\aj\bi)^\perp\|^{s_0-1}\int_{\proj_{V(\aj\bi)}(X)} \cH^{s_0-1}_\infty\left(X\cap\proj_{V(\aj\bi)}^{-1}(t)\right)d\lambda(t)\\
	&=(\cL^nh)(\bi)=h(\bi).
	\end{align*}
	Hence,
	\begin{align*}
	h(\bi)&\geq\liminf_{\delta\to0}\int_{\proj_{V(\bi)}(X)} \cH^{s_0-1}_\delta\left(X\cap\proj_{V(\bi)}^{-1}(t)\right)d\lambda(t)\\
	&\geq\int_{\proj_{V(\bi)}(X)} \liminf_{\delta\to0}\cH^{s_0-1}_\delta\left(X\cap\proj_{V(\bi)}^{-1}(t)\right)d\lambda(t)\\
	&=\int_{\proj_{V(\bi)}(X)}\cH^{s_0-1}\left(X\cap\proj_{V(\bi)}^{-1}(t)\right)d\lambda(t)\\
	&\geq \int_{\proj_{V(\bi)}(X)}\cH^{s_0-1}_\infty\left(X\cap\proj_{V(\bi)}^{-1}(t)\right)d\lambda(t)=h(\bi),
	\end{align*}
which completes the proof.
\end{proof}

In particular, we get that
\begin{equation}\label{eq:contismeas}
\int \cH_\infty^{s_0-1}\left(X\cap\proj_{V(\bj)}^{-1}(t)\right)d\lambda(t)=\int \cH^{s_0-1}\left(X\cap\proj_{V(\bj)}^{-1}(t)\right)d\lambda(t)
\end{equation}
for every $\bj\in\Sigma$, hence, the right-hand side is always finite. This has the following simple consequence:

\begin{lemma}\label{lem:contentismeas}
	Let $B\subseteq X$  be a Borel set. Then for every $\bj\in\Sigma$
	$$
	\int \cH_\infty^{s_0-1}\left(B\cap\proj_{V(\bj)}^{-1}(t)\right)d\lambda(t)=\int \cH^{s_0-1}\left(B\cap\proj_{V(\bj)}^{-1}(t)\right)d\lambda(t).
	$$
	In particular, for every Borel subset $B\subset X$ and every $\bj\in\Sigma$, $\cH_\infty^{s_0-1}\left(B\cap\proj_{V(\bj)}^{-1}(t)\right)=\linebreak\cH^{s_0-1}\left(B\cap\proj_{V(\bj)}^{-1}(t)\right)$ for $\lambda$-almost every $t$.
\end{lemma}
{
\begin{proof}
	Since the Hausdorff content of a set always bounds the Hausdorff measure from below, we get
	\begin{align*}
		\int \cH_\infty^{s_0-1}\left(B\cap\proj_{V(\bj)}^{-1}(t)\right)d\lambda(t)&\leq\int \cH^{s_0-1}\left(B\cap\proj_{V(\bj)}^{-1}(t)\right)d\lambda(t)\\
		\intertext{then by using the additivity of the Hausdorff measure, we get}
		&=\int\cH^{s_0-1}\left(X\cap\proj_{V(\bj)}^{-1}(t)\right)-\cH^{s_0-1}\left((X\setminus B)\cap\proj_{V(\bj)}^{-1}(t)\right)d\lambda(t)\\
		\intertext{now by \eqref{eq:contismeas}}
		&\leq\int\cH^{s_0-1}_\infty\left(X\cap\proj_{V(\bj)}^{-1}(t)\right)-\cH^{s_0-1}_\infty\left((X\setminus B)\cap\proj_{V(\bj)}^{-1}(t)\right)d\lambda(t)\\
		&\leq \int \cH_\infty^{s_0-1}\left(B\cap\proj_{V(\bj)}^{-1}(t)\right)d\lambda(t),
	\end{align*}
where the last inequality uses the countable subadditivity of the Hausdorff content.
\end{proof}}

Another important corollary of \cref{prop:Hausmeas} is the following:

\begin{lemma}\label{lem:intersection0}
	For $\bj\neq\hbar\in\Sigma_*$ with $[\bj]\cap[\hbar]=\emptyset$, and $\bi\in\Sigma$,
	$$
	\int \cH^{s_0-1}\left(f_{\bj}(X)\cap f_{\hbar}(X)\cap\proj_{V(\bi)}^{-1}(t)\right)d\lambda(t)=0.
	$$
	In particular, for every $\cH^{s_0-1}\left(f_{\bj}(X)\cap f_{\hbar}(X)\cap\proj_{V(\bi)}^{-1}(t)\right)=0$ for $\lambda$-almost every $t$.
\end{lemma}

\begin{proof}
	It is enough to show the claim of the lemma for finite words with equal length. Thus, similarly to the previous arguments, for every $n\geq1$, $\bj,\hbar\in\Sigma_n$ and $\bi\in\Sigma$
	\begin{align*}
		h(\bi)&=\int\cH^{s_0-1}\left(X\cap\proj_{V(\bi)}^{-1}(t)\right)d\lambda(t)\\
		&=\int\cH^{s_0-1}\left(\bigcup_{\bj'\in\Sigma_n}f_{\bj'}(X)\cap\proj_{V(\bi)}^{-1}(t)\right)d\lambda(t)\\
		&\leq\sum_{|\bj'|=n}\int\cH^{s_0-1}\left(f_{\bj'}(X)\cap\proj_{V(\bi)}^{-1}(t)\right)d\lambda(t)-\int \cH^{s_0-1}\left(f_{\bj}(X)\cap f_{\hbar}(X)\cap\proj_{V(\bi)}^{-1}(t)\right)d\lambda(t)\\
		&=\sum_{|\bj'|=n}\|A_{\bj'}^*|V(\bi)\|\|A_{\bj'}|V(\aj'\bi)^\perp\|^{s_0-1}\int_{\proj_{V(\aj'\bi)}(X)} \cH^{s_0-1}\left(X\cap\proj_{V(\aj'\bi)}^{-1}(t)\right)d\lambda(t)\\
		&\qquad\qquad-\int \cH^{s_0-1}\left(f_{\bj}(X)\cap f_{\hbar}(X)\cap\proj_{V(\bi)}^{-1}(t)\right)d\lambda(t)\\
		&=(\cL^nh)(\bi)-\int \cH^{s_0-1}\left(f_{\bj}(X)\cap f_{\hbar}(X)\cap\proj_{V(\bi)}^{-1}(t)\right)d\lambda(t)\\
		&=h(\bi)-\int \cH^{s_0-1}\left(f_{\bj}(X)\cap f_{\hbar}(X)\cap\proj_{V(\bi)}^{-1}(t)\right)d\lambda(t),
	\end{align*}
	where we applied \cref{prop:Hausmeas} and \cref{prop:Hauscontent}.
\end{proof}

\begin{lemma}\label{lem:tomeas2} For every $k\in\cA$ and every Borel set $B\subseteq X$,
	$$
	\int \cH^{s_0-1}(f_{k}(B)\cap \proj_{V(\bi)}^{-1}(t))d\lambda(t)=\|A_{k}^*|V(\bi)\|\|A_{k}|V(k\bi)^\perp\|^{s_0-1}\int \cH^{s_0-1}\left(B\cap\proj_{V(k\bi)}^{-1}(t)\right)d\lambda(t).
	$$
\end{lemma}

\begin{proof} Using the facts that $F_{k,V(\bi)}\colon\bbR\to\bbR$ and $f_k\colon V(k\bi)^\perp\to V(\bi)^\perp$ are affine maps, we get by simple algebraic manipulations that
	\begin{align*}
	\int \cH^{s_0-1}&\left(f_{k}(B)\cap \proj_{V(\bi)}^{-1}(t)\right)d\lambda(t)\\
	&=\int_{\proj_{V(\bi)}(f_k(B))} \cH^{s_0-1}\left(f_k(B)\cap\proj_{V(\bi)}^{-1}(t)\right)d\lambda(t)\\
	&=\int_{F_{k,V(\bi)}(\proj_{V(k\bi)}(B))} \cH^{s_0-1}\left(f_k(B)\cap\proj_{V(\bi)}^{-1}(t)\right)d\lambda(t)\\
	&=\|A_k^*|V(\bi)\|\int_{\proj_{V(k\bi)}(B)} \cH^{s_0-1}\left(f_k(B)\cap\proj_{V(\bi)}^{-1}(F_{k,V(\bi)}(t))\right)d\lambda(t)\\
	&=\|A_k^*|V(\bi)\|\int_{\proj_{V(k\bi)}(B)} \cH^{s_0-1}\left(f_k(B\cap\proj_{V(k\bi)}^{-1}(t))\right)d\lambda(t)\\
	&=\|A_k^*|V(\bi)\|\|A_k|V(k\bi)^\perp\|^{s_0-1}\int_{\proj_{V(k\bi)}(B)} \cH^{s_0-1}\left(B\cap\proj_{V(k\bi)}^{-1}(t)\right)d\lambda(t)\\
	&=\|A_k^*|V(\bi)\|\|A_k|V(k\bi)^\perp\|^{s_0-1}\int\cH^{s_0-1}\left(B\cap\proj_{V(k\bi)}^{-1}(t)\right)d\lambda(t).
\end{align*}
\end{proof}

\subsection{An alternative form of the K\"aenm\"aki measure}

For every $\bi\in\Sigma$, let us define a measure on $\Sigma$ as follows: for every $\bj\in\Sigma_*$
$$
\eta_\bi([\bj]):=\int \cH^{s_0-1}\left(f_{\bj}(X)\cap\proj_{V(\bi)}^{-1}(t)\right)d\lambda(t).
$$
First, we will show that $\eta_\bi$ can be extended to a well-defined Borel measure on $\Sigma$. (Note that $\eta_\bi$ might be the zero measure.) To do so, it is enough to show the following lemma:

\begin{lemma}\label{prop:etameas}
	For every $\bj\in\Sigma_*$, and $\bi\in\Sigma$
	$$
	\eta_{\bi}([\bj])=\sum_{k\in\cA}\eta_{\bi}([\bj k]).
	$$
\end{lemma}

\begin{proof}
	By \cref{lem:intersection0} and \cref{lem:tomeas2}, it follows that
	\begin{align*}
		\sum_{k\in\cA}\eta_\bi([\bj k])&=\sum_{k\in\cA}\int \cH^{s_0-1}(f_{\bj k}(X)\cap \proj_{V(\bi)}^{-1}(t))d\lambda(t)\\
		&= \|A_{\bj}^*|V(\bi)\|\|A_{\bj}|V(\aj\bi)^\perp\|^{s_0-1}\sum_{k\in\cA}\int\cH^{s_0-1}(f_{k}(X)\cap \proj_{V(\aj\bi)}^{-1}(t))d\lambda(t)\\
		&=\|A_{\bj}^*|V(\bi)\|\|A_{\bj}|V(\aj\bi)^\perp\|^{s_0-1}\int\cH^{s_0-1}(X\cap \proj_{V(\aj\bi)}^{-1}(t))d\lambda(t)\\
		&=\int \cH^{s_0-1}(f_{\bj}(X)\cap \proj_{V(\bi)}^{-1}(t))d\lambda(t)=\eta_{\bi}([\bj]).
	\end{align*}
\end{proof}

{ Since} $\eta_{\bi}$ is a Borel measure on $\Sigma$, by \cref{lem:intersection0} and the fact that the Borel $\sigma$-algebra on $X$ is the smallest $\sigma$-algebra generated by the sets $\{f_{\bi}(X)\}_{\bi\in\Sigma_*}$, we get that for every $\bi\in\Sigma$ and every Borel subset $B\subseteq X$
\begin{equation}\label{eq:etameas}
	\pi_*\eta_{\bi}(B)=\int\cH^{s_0-1}(B\cap\proj_{V(\bi)}^{-1}(t))d\lambda(t).
\end{equation}

Now, we show the dichotomy that $\eta_\bi$ is either trivial for every $\bi\in\Sigma$, i.e. it is the uniformly zero measure or it is uniformly equivalent to the K\"aenm\"aki measure for every $\bi\in\Sigma$.

\begin{proposition}\label{prop:equivKaenmaki}
	For every $\bi\in\Sigma$, the measure $\eta_{\bi}$ is not the uniformly zero measure on $\Sigma$ if and only if $\inf_{\bi\in\Sigma}\int\cH^{s_0-1}(X\cap\proj_{V(\bi)}^{-1}(t))d\lambda(t)>0$.
	
	Moreover, if  $\inf_{\bi\in\Sigma}\int\cH^{s_0-1}(X\cap\proj_{V(\bi)}^{-1}(t))d\lambda(t)>0$ then there exists a constant $C>0$ such that for every $\bj\in\Sigma_*$ and $\bi\in\Sigma$
	$$
	C^{-1}\mu_K([\bj])\leq\eta_{\bi}([\bj])\leq C\mu_K([\bj]).
	$$
\end{proposition}

\begin{proof}
	Observe that by \cref{lem:tomeas2} and the combination of \cref{prop:Hauscontent} and \cref{prop:Hausmeas}, we get
	\begin{align*}
	\eta_{\bi}([\bj])&=\|A_{\bj}^*|V(\bi)\|\|A_{\bj}|V(\aj\bi)^\perp\|^{s_0-1}\int\cH^{s_0-1}(X\cap \proj_{V(\aj\bi)}^{-1}(t))d\lambda(t)\\
	&=\|A_{\bj}^*|V(\bi)\|\|A_{\bj}|V(\aj\bi)^\perp\|^{s_0-1}\int\cH^{s_0-1}_\infty(X\cap \proj_{V(\aj\bi)}^{-1}(t))d\lambda(t)\\
	&\leq \|A_{\bj}^*|V(\bi)\|\|A_{\bj}|V(\aj\bi)^\perp\|^{s_0-1}|X|^{s_0}\\
	&\leq C\alpha_1(A_{\bj})\alpha_2(A_{\bj})^{s_0-1}\leq C'\mu_K([\bj]),
	\end{align*}
	where in the last two inequalities we used \cref{eq:Kaenmakimeas} and \cref{eq:domin}. Similarly,
	\begin{align*}
		\eta_{\bi}([\bj])&=\|A_{\bj}^*|V(\bi)\|\|A_{\bj}|V(\aj\bi)^\perp\|^{s_0-1}\int\cH^{s_0-1}_\infty(X\cap \proj_{V(\aj\bi)}^{-1}(t))d\lambda(t)\\
		&=\|A_{\bj}^*|V(\bi)\|\|A_{\bj}|V(\aj\bi)^\perp\|^{s_0-1}\cdot p(\aj\bi)\cdot\iint\cH^{s_0-1}_\infty(X\cap \proj_{V(\bi)}^{-1}(t))d\lambda(t)d\nu(\bi)\\
		&\geq C''\mu_K([\bj])\cdot\iint\cH^{s_0-1}_\infty(X\cap \proj_{V(\bi)}^{-1}(t))d\lambda(t)d\nu(\bi)\cdot\inf_{\bi\in\Sigma}p(\bi).
	\end{align*}
Now, $\iint\cH^{s_0-1}_\infty(X\cap \proj_{V(\bi)}^{-1}(t))d\lambda(t)d\nu(\bi)>0$ if and only if $\inf\limits_{\bi\in\Sigma}\int\cH^{s_0-1}_\infty(X\cap \proj_{V(\bi)}^{-1}(t))d\lambda(t)>0$ by \cref{prop:Hauscontent}, which completes the proof.
\end{proof}

Now, we consider a more sophisticated version of \cref{eq:Kaenmakimeas}.

\begin{proposition}\label{prop:equivKaenmaki2}
	If  $\inf_{\bi\in\Sigma}\int\cH^{s_0-1}(X\cap\proj_{V(\bi)}^{-1}(t))d\lambda(t)>0$ then for every Borel subset $B\subseteq X$
	$$
	\pi_*\mu_K(B)=\frac{\iint\cH^{s_0-1}(B\cap\proj_{V(\bi)}^{-1}(t))d\lambda(t)d\nu(\bi)}{\iint\cH^{s_0-1}(X\cap\proj_{V(\bi)}^{-1}(t))d\lambda(t)d\nu(\bi)}.
	$$
\end{proposition}

\begin{proof}
	First, we will show that the K\"aenm\"aki measure $\mu_K$ equals to the measure $\gamma:=\tfrac{\int\eta_\bi d\nu(\bi)}{\int\eta_\bi(X)d\nu(\bi)}$ on $\Sigma$. By \cref{eq:Kaenmakimeas} and \cref{prop:equivKaenmaki}, $\gamma$ is equivalent to $\mu_K$, and so, it is enough to show that $\gamma$ is $\sigma$-invariant. Indeed, if $B$ is such that $\sigma^{-1}B=B$ then either $\mu_K(B)=0$ or $\mu_K(B^c)=0$, but then by \cref{eq:Kaenmakimeas}, either $\gamma(B)=0$ or $\gamma(B^c)=0$, which implies the ergodicity of $\gamma$, and since ergodic probability measures are either singular or equal, the claim follows.
	
	The invariance is enough to be verified over cylinder sets. For simplicity, let us denote for a finite word $\bj\in\Sigma_*$ the function $\bi\mapsto\int\cH^{s_0-1}\left(f_{\bj}(X)\cap\proj_{V(\bi)}^{-1}(t)\right)d\lambda(t)$ by $\tilde{h}_{\bj}(\bi)$. Thus, by \cref{lem:tomeas2}
	\begin{align*}
	\sum_{k\in\cA}\int\tilde{h}_{k\bj}(\bi)d\nu(\bi)=&\sum_{k\in\cA}\iint\cH^{s_0-1}(f_{k\bj}(X)\cap\proj_{V(\bi)}^{-1}(t))d\lambda(t)d\nu(\bi)\\
		&=\int\sum_{k\in\cA}\|A_k^*|V(\bi)\|\|A_k|V(k\bi)^\perp\|^{s_0-1}\int\cH^{s_0-1}\left(f_{\bj}(X)\cap\proj_{V(k\bi)}^{-1}(t)\right)d\lambda(t)d\nu(\bi)\\
		&=\int(\cL \tilde{h}_{\bj})(\bi)d\nu(\bi)=\int \tilde{h}_{\bj}(\bi)d(\cL^*\nu)(\bi)=\int \tilde{h}_{\bj}(\bi)d\nu(\bi).
	\end{align*}
The claim follows then by \cref{eq:etameas}.
\end{proof}

\section{Characterisation of positive measure}

This section is devoted to prove our main theorems. Let us note that Marstrand \cite{Marstrand1954b} showed that for any Borel subset $E\subset\bbR$ and every subspace $V\in\bbRP^1$
$$
\cH^{s}(E)\geq\int_{\proj_V(E)}\cH^{s-1}(E\cap\proj_V^{-1}(t))d\lambda_V(t).
$$
Hence, item \eqref{it:posintexist} implies item \eqref{it:posmeas} in \cref{thm:main}. Our first main lemma shows that a kind of reversed inequality holds for self-affine sets.

\begin{lemma}\label{lem:hausintform}
	Let $\Phi=\{f_i(x)=A_ix+t_i\}_{i\in\cA}$ be a dominated planar IFS of affinities. Let $X$ be the attractor of $\Phi$ and let $s_0\in(1,2]$ be the affinity dimension. Then there exists a constant $C>0$ such that
$$
\cH^{s_0}(X)\leq C\max_{\bi\in\Sigma}\int\cH^{s_0-1}_\infty\left(X\cap\proj_{V(\bi)}^{-1}(t)\right)d\lambda(t).
$$	
\end{lemma}

\begin{proof}
{	Let $\epsilon>0$ be arbitrary but fixed. By \cref{lem:semicont0}, the map $(V,t)\mapsto\cH_\infty^{s_0-1}\left(X\cap\proj_{V}^{-1}(t)\right)$ is upper semi-continuous, and so, by \cite[Theorem~2.1.3]{Ransford} and the monotone convergence theorem, there exists a continuous function $a_\bi\colon\bbR\to\bbR$ such that $\cH_\infty^{s_0-1}\left(X\cap\proj_{V(\bi)}^{-1}(t)\right)\leq a_\bi(t)$ and\linebreak $\int a_\bi(t)d\lambda(t)\leq\int\cH^{s_0-1}_\infty\left(X\cap\proj_{V(\bi)}^{-1}(t)\right)d\lambda(t)+\epsilon$. We may choose $a_{\bi}$ such that its support is contained in $[-1,1]$, and so, there exists $\delta(\bi)>0$ such that for every $t,t'\in\bbR$ if $|t-t'|<\delta$ then $|a_\bi(t)-a_{\bi}(t')|<\epsilon$.}
	
	For every $(\bi,t)\in\Sigma\times\bbR$, let $\{U_{i,t}\}_{i\in\cI_{\bi,t}}$ be a cover of $X\cap\proj_{V(\bi)}^{-1}(t)$ by open intervals in $\proj_{V(\bi)}^{-1}(t)$ such that $\sum_{i\in\cI_{\bi,t}}|U_{i,t}|^{s_0-1}\leq\cH^{s_0-1}_\infty(X\cap \proj_{V(\bi)}^{-1}(t))+\epsilon$. By the compactness, we may assume that $\cI(\bi,t)$ is finite. Then for every $(\bi,t)$, there exists $r(\bi,t)>0$ such that for every $|t-t'|<r(\bi,t)$, $X\cap \proj_{V(\bi)}^{-1}(t')\subseteq\bigcup_{i\in\cI_{\bi,t}}U_{i,t}$. We may also assume that $r(\bi,t)\leq\delta(\bi)$ by possibly taking minimum.
	
	By applying Besicovitch's covering theorem, there exists a $Q\geq1$ (independent of the quantities above) such that there exists $\cB_1(\bi),\ldots,\cB_Q(\bi)$ collection of points such that
	\begin{itemize}
		\item $\proj_{V(\bi)}(X)\subseteq\bigcup_{i=1}^Q\bigcup_{t\in\cB_i(\bi)}B(t,r(\bi,t))$,
		\item $B(t,r(\bi,t))\cap B(t',r(\bi,t'))=\emptyset$ for every $i=1,\ldots,Q$ and $t\neq t'\in\cB_i(\bi)$.
	\end{itemize}	
	Since $\proj_{V(\bi)}(X)$ is compact, there exists finite subsets $\cB_i'(\bi)\subseteq\cB_i(\bi)$ such that $\proj_{V(\bi)}(X)\subseteq\bigcup_{i=1}^Q\bigcup_{t\in\cB_i'(\bi)}B(t,r(\bi,t))$. Now, since $\bigcup_{i=1}^Q\cB_i'(\bi)$ is finite there exists $N=N(\bi)$ such that for every $n\geq N(\bi)$
	$$
{	\frac{\|A_{\ai|_n}|V(\bi)^\perp\|}{\|A_{\ai|_n}|V(\bi)\|}}\cdot|X|\leq\min_{t\in\bigcup_{i=1}^Q\cB_i'}r(\bi,t),
	$$
	where we recall that $\ai|_n=(i_n,\ldots,i_1)$ for $\bi=(i_1,i_2,\ldots)$. For every $t\in\bigcup_{i=1}^Q\cB_i'(\bi)$, and $j\in\cI(\bi,t)$ let $\tilde{U}_{t,i}=U_{i,t}\times B(t,r(\bi,t))$ be the rectangle, axes parallel to $V(\bi)$ and $V(\bi)^\perp$. By the construction, $\bigcup_{t\in\bigcup_{i=1}^Q\cB_i'(\bi)}\bigcup_{i\in\cI_{\bi,t}}\tilde{U}_{t,i}$ is a cover of $X$.

	{ Unfortunately, the function $\bi\mapsto N(\bi)$ is not necessarily measurable. However, the function $N_m(\bi):=\min\{N(\bj):\bj\in[\bi|_m]\}$ is constant over $m$-length cylinder sets, and thus, $M(\bi):=\limsup_{m\to\infty}N_m(\bi)$ is Borel measurable. Furthermore, $M(\bi)\leq N(\bi)$ for every $\bi\in\Sigma$, so $\bigcup_{k=1}^\infty\{\bi:M(\bi)\leq k\}=\Sigma$. Therefore, there exists $M\geq1$ such that
	$$
	\epsilon>\mu_F(\{\bi:M(\bi)> M\})=\mu_F\left(\bigcap_{m=1}^\infty\bigcup_{p=m}^\infty\{\bi:N_p(\bi)> M\}\right)\geq\lim_{m\to\infty}\mu_F(\{\bi:N_m(\bi)> M\})
	$$
	Hence, there exists $M'\geq M$ such that for every $m\geq M'$
	\begin{equation}\label{eq:smallmeas}
	\epsilon>\mu_F(\{\bi:N_m(\bi)> M\})\geq\mu_F(\{\bi:N_m(\bi)> m\}).
	\end{equation}

Now, we will construct our cover with diameters at most $(\max_i\|A_i\|)^m\cdot|X|$. Let $\cG_m:=\{\bi\in\Sigma_m:\text{ there exists }\bj\in[\bi]\text{ such that }N(\bj)\leq m\}$. By the definition of $N_m$, $\bigcup_{\bi\in\mathcal{G}_m^c}[\bi]=\{\bi\in\Sigma:N_m(\bi)>m\}$. Thus, $\mu_F(\bigcup_{\bi\in\mathcal{G}_m^c}[\bi])\leq\epsilon$ by \eqref{eq:smallmeas}. For every $\bi\in\cG_m$, let $\bi'\in\Sigma$ be such that $N(\bi')\leq m$ and $\bi'\in[\bi|_m]$.}
	
	For every $\bi\in\cG_m^c$, let us cover $f_{\ai}(X)$ with $\lceil\alpha_1(A_\ai)/\alpha_2(A_\ai)\rceil$-many rectangles with side length $\alpha_2(A_\ai)|X|$. For every $\bi\in\cG_m$, cover the parallelogram $f_{\ai}(\tilde{U}_{t,i})$ with $\left\lceil\tfrac{2\|A_{\ai}|V(\bi')\|r(\bi',t)}{\|A_{\ai}|V(\bi')^\perp\|\cdot |U_{i,t}|}\right\rceil$-many lozenge being axes parallel to the original with side length $\|A_{\ai}|V(\bi')^\perp\|\cdot |U_i|$. Since the system is dominated, { $A_{\ai}V(\bi')^\perp=V(\sigma^m\bi')^\perp$ and $A_{\ai}V(\bi')$} are uniformly transverse and there exists a constant $c>0$ (independent of the quantities above) the diameter of such lozenge is at most $c\|A_{\ai}|V(\bi')\|\cdot |U_i|$.
	
	Hence,
	\begin{align*}
		\cH^{s_0}_{(\max_i\|A_i\|)^m|X|}(X)&\leq\sum_{\bi\in\cG_m^c}\left\lceil\frac{\alpha_1(A_\ai)}{\alpha_2(A_\ai)}\right\rceil\left(\alpha_2(A_\ai)|X|\right)^{s_0}\\
		&\qquad\qquad+\sum_{\bi\in\cG_m}\sum_{i=1}^Q\sum_{t\in\bigcup\cB_i'(\bi')}\sum_{j\in\cI_{\bi',t}}\left\lceil\frac{2\|A_{\ai}|V(\bi')\|\cdot r(\bi',t)}{\|A_{\ai}|V(\bi')^\perp\|\cdot |U_{j,t}|}\right\rceil\left(c\|A_{\ai}|V(\bi')^\perp\|\cdot |U_{j,t}|\right)^{s_0}\\
		&\lesssim \mu_F\left(\bigcup_{\bi\in\cG_m^c}[\bi]\right)+\sum_{\bi\in\cG_m}\alpha_1(A_\ai)\alpha_2(A_\ai)^{s_0-1}\sum_{i=1}^Q\sum_{t\in\cB_i'(\bi')}r(\bi',t)\sum_{j\in\cI_{\bi',t}}|U_{j,t}|^{s_0-1}\\
		&\lesssim \epsilon+\sum_{\bi\in\cG_m}\alpha_1(A_\ai)\alpha_2(A_\ai)^{s_0-1}\sum_{i=1}^Q\sum_{t\in\cB_i'(\bi')}r(\bi',t)\left(\cH^{s_0-1}_\infty(\proj_{V(\bi')}^{-1}(t)\cap X)+\epsilon\right)\\
		&\leq\epsilon+\sum_{\bi\in\cG_m}\alpha_1(A_\ai)\alpha_2(A_\ai)^{s_0-1}\sum_{i=1}^Q\sum_{t\in\cB_i'(\bi')}r(\bi',t)\left(a_{\bi'}(t)+\epsilon\right)
		\intertext{by using that $r(\bi,t)\leq\delta(\bi)$ and the balls in $\cB_i'(\bi')$ are disjoint we get}
		&\leq\epsilon+\sum_{\bi\in\cG_m}\alpha_1(A_\ai)\alpha_2(A_\ai)^{s_0-1}Q\int\left(a_{\bi'}(t)+2\epsilon\right)d\lambda(t)\\
		&\leq\epsilon+\sum_{\bi\in\cG_m}\alpha_1(A_\ai)\alpha_2(A_\ai)^{s_0-1}Q\left( \epsilon(2|X|+1)+\int\cH^{s_0-1}_\infty\left(X\cap\proj_{V(\bi')}^{-1}(t)\right)d\lambda(t)\right)\\
		&\lesssim \epsilon+Q\left(\max_{\bi\in\Sigma}\int\cH^{s_0-1}_\infty\left(X\cap\proj_{V(\bi)}^{-1}(t)\right)d\lambda(t)+(2|X|+1)\epsilon\right)\cdot\mu_F\left(\bigcup_{\bi\in\cG_m}[\bi]\right),\\
	\end{align*}
	where we applied \cref{eq:Kaenmakimeas} many times and the assumption on the diameters $r(\bi,t)$. Since $m$ was arbitrary above, we get
	$$
	\cH^{s_0}(X)\lesssim \epsilon+\max_{\bi\in\Sigma}\int\cH^{s_0-1}_\infty\left(X\cap\proj_{V(\bi)}^{-1}(t)\right)d\lambda(t).
	$$
Since $\epsilon>0$ was arbitrary, the claim follows.	
\end{proof}

\begin{proof}[Proof of \cref{thm:main}]
	The implication \eqref{it:posmeas}$\Rightarrow$\eqref{it:posintexist} follows by \cref{lem:hausintform}. The equivalence \eqref{it:posintexist}$\Leftrightarrow$\eqref{it:posintevery} follows by \cref{prop:Hauscontent}.
	
	The implication \eqref{it:posintevery}$\Rightarrow$\eqref{it:massdist} follows by \cref{lem:contentismeas} and \cref{prop:equivKaenmaki2}. The implication \eqref{it:massdist}$\Rightarrow$\eqref{it:posmeas} follows by the mass distribution principle, see for example \cite[Theorem~4.2]{Falconer1990}.
\end{proof}

Now, we study the consequences of positive Hausdorff measure, and prove \cref{thm:maincor}.

\begin{proof}[Proof of \cref{thm:maincor}]
	First, we show that $\cH^{s_0}(X)>0$ implies \eqref{it:boundeddens}. Let $x\in\bbR$ and $r>0$ be arbitrary { and choose $n\geq1$ such that $\alpha_1(A_{\bi})|X|<r$ for every $\bi\in\Sigma_n$}. Then for every $\bj\in\Sigma$
	\begin{align*}
		(\proj_{V(\bj)})_*\pi_*\mu_K(B(x,r))&\leq\sum_{\substack{|\bi|=n\\f_{\bi}(X)\cap \proj_{V(\bj)}^{-1}(B(x,r))\neq\emptyset}}\mu_K([\bi])\\
		&\leq C^{-1}\sum_{\substack{|\bi|=n\\f_{\bi}(X)\cap \proj_{V(\bj)}^{-1}(B(x,r))\neq\emptyset}}\int \cH^{s_0-1}\left(f_{\bi}(X)\cap\proj_{V(\bj)}^{-1}(t)\right)d\lambda(t)
		\intertext{ by \cref{thm:main} and \cref{prop:equivKaenmaki}}
		&= C^{-1}\int \cH^{s_0-1}\left(\bigcup_{\substack{|\bi|=n\\f_{\bi}(X)\cap \proj_{V(\bj)}^{-1}(B(x,r))\neq\emptyset}}f_{\bi}(X)\cap\proj_{V(\bj)}^{-1}(t)\right)d\lambda(t)\text{ by \cref{lem:intersection0}}\\
		&\leq C^{-1}\int \cH^{s_0-1}\left(\proj_{V(\bj)}^{-1}(B(x,2r))\cap\proj_{V(\bj)}^{-1}(t)\right)d\lambda(t)\\
		&=C^{-1}\int \cH^{s_0-1}_\infty\left(\proj_{V(\bj)}^{-1}(B(x,2r))\cap\proj_{V(\bj)}^{-1}(t)\right)d\lambda(t)\leq C^{-1}|X|^{s_0-1}2r,
	\end{align*}
	where in the last equality we used \cref{lem:contentismeas}.
	
	Now, let us prove \eqref{it:boundedmeas}. For $r>0$, let $\Gamma_r=\{\bi\in\Sigma_*:\alpha_1(A_\bi)|X|\leq r<\alpha_1(A_{\bi_-})|X|\}$. Let $\bi\in\Sigma$, $r>0$ and $t\in\bbR$ be arbitrary. Then
	\begin{align*}
		\cH^{s_0-1}_r(X\cap\proj_{V(\bi)}^{-1}(t))&\leq\sum_{\substack{\bj\in\Gamma_r\\f_{\bj}(X)\cap\proj_{V(\bi)}^{-1}(t)\neq\emptyset}}|f_{\bj}(X)\cap\proj_{V(\bi)}^{-1}(t)|^{s_0-1}\\
		&\leq \sum_{\substack{\bj\in\Gamma_r\\f_{\bj}(X)\cap\proj_{V(\bi)}^{-1}(t)\neq\emptyset}}\|A_{\bj}|V(\aj\bi)^\perp\|^{s_0-1}|X|^{s_0-1}\\
		&\lesssim r^{-1}\sum_{\substack{\bj\in\Gamma_r\\f_{\bj}(X)\cap\proj_{V(\bi)}^{-1}(t)\neq\emptyset}}\alpha_1(A_\bj)\alpha_2(A_\bj)^{s_0-1}\\
		&{\lesssim r^{-1}(\proj_{V(\bi)})_*\pi_*\mu_K(B(t,r))}\leq C,
	\end{align*}
	where the last inequality follows by \eqref{it:boundeddens}. Since $r>0$ was arbitrary, we get that\linebreak $\cH^{s_0-1}(X\cap\proj_{V(\bi)}^{-1}(t))\leq C$ for every $\bi\in\Sigma$ and $t\in\bbR$.
\end{proof}

Finally, we show the equivalence of the positive measure with the uniformly bounded density of the projection of the K\"aenm\"aki measure.

\begin{proof}[Proof of \cref{thm:mainsep}]
The direction \eqref{it:posmeas2}$\Rightarrow$\eqref{it:boundeddens1} follows by \cref{thm:maincor}, so it is enough to show the implication \eqref{it:boundeddens1}$\Rightarrow$\eqref{it:massdist} of \cref{thm:main}.

For $r>0$, let us recall the definition of $\Delta_r$ from \cref{eq:Delta}. Let $x\in X$ be arbitrary. Then
\begin{align*}
	\pi_*\mu_K(B(x,r))&\leq\sum_{\substack{\bi\in\Delta_r\\f_{\bi}(X)\cap B(x,r)\neq\emptyset}}\pi_*\mu_K(B(x,r)\cap f_{\bi}(X))\\
	&\leq C \sum_{\substack{\bi\in\Delta_r\\f_{\bi}(X)\cap B(x,r)\neq\emptyset}}\alpha_1(A_\bi)\alpha_2(A_\bi)^{s_0-1}\pi_*\mu_K(f_{\bi}^{-1}(B(x,r)\cap f_{\bi}(X)))\\
	&\leq C \sum_{\substack{\bi\in\Delta_r\\f_{\bi}(X)\cap B(x,r)\neq\emptyset}}\alpha_1(A_\bi)\alpha_2(A_\bi)^{s_0-1}\pi_*\mu_K(\proj_{V(\ai\bj)}^{-1}(B(x,\frac{r}{\|A_{\bi}^*|V(\bj)\|})))\\
	&\leq C'\sum_{\substack{\bi\in\Delta_r\\f_{\bi}(X)\cap B(x,r)\neq\emptyset}}\alpha_1(A_\bi)r^{s_0-1}\frac{r}{\alpha_1(A_{\bi})}\leq C''r^{s_0}.
\end{align*}\end{proof}

\begin{proof}[Proof of \cref{cor}]
	Suppose that $\cH^{s_0}(X)>0$ and the IFS satisfies the bounded neighbourhood condition. By \cite[Proposition~3.1]{AnttilaBaranyKaenmaki24}, $$\dim_A X\leq 1+\max_{V\in X_F}\max_{t\in\proj_V(X)}\dim_H(X\cap\proj_V^{-1}(t)).$$ By \cref{thm:maincor}, $\dim_H(X\cap\proj_V^{-1}(t))\leq s_0-1$ for every $V\in X_F$ and $t\in\proj_V(X)$. Since $\dim_AX\geq \dimh X=s_0$, the claim follows.
\end{proof}

\section{Verification of the examples}\label{sec:examples}

Our final section is devoted to verify the examples presented in \cref{subsec:examples}. Our strategy is the following: we give conditions under which the planar system satisfies the strong separation condition and hence, the open bounded neighbourhood condition, and then we show that the projections of the K\"aenm\"aki measure along Furstenberg directions are absolutely continuous with continuous density. To show this, we borrow Fourier analytic methods from Feng and Feng \cite{FengFeng}.

For a Borel probability measure $\eta$ on $\bbR^d$, let us denote by $\widehat{\eta}\colon\bbR^d\to\bbC$ the Fourier transform of $\eta$, that is,
$$
\widehat{\eta}(\xi)=\int e^{i\langle\xi,x\rangle}d\eta(x).
$$
By \cite[Theorem~5.4]{Mattila15}, if there exists a $t>d$ such that
\begin{equation}\label{eq:enough}
\int|\widehat{\eta}(\xi)|^2\|\xi\|^td\xi<\infty
\end{equation}
then $\eta\ll\cL_d$ with continuous density.

\subsection{Diagonal example}

Before we verify \cref{ex:singlexf}, we need the following lemma. Although, we believe that this lemma is well-known, we could not find any proper reference.

\begin{lemma}\label{lem:fourier1}
	Let $\{x\mapsto c_ix+\tau_i\}_{i\in\cA}$ be a self-similar IFS on the real line with natural projection $\pi_\tau$ and let $(p_i)_{i\in\cA}$ be a probability vector and $\nu$ be the corresponding Bernoulli measure on $\Sigma$. If $\max_{i\in\cA}|c_i|<1/2$ and $\sum_{i\in\cA}\left(\frac{p_i}{|c_i|}\right)^2<1$ then the self-similar measure $\eta_\tau=(\pi_\tau)_*\nu$ is absolutely continuous with continuous density for Lebesgue-almost every $\tau:=(\tau_i)_{i\in\cA}$.
\end{lemma}

For simplicity, let $a_{\bi}=a_{i_1}\cdots a_{i_n}$ for $\bi\in\Sigma_*$. Let us write $\pi_\tau$ for the natural projection of $\{x\mapsto c_ix+\tau_i\}_{i\in\cA}$. Then
$$
\pi_\tau(\bi)=\sum_{k=1}^\infty \tau_{i_k}c_{\bi|_{k-1}}=\sum_{j\in\cA}\tau_i\sum_{k=1}^\infty\delta_{i_k}^jc_{\bi|_{k-1}},
$$
where $\delta_i^j=1$ if $i=j$ and otherwise $0$. Let $\Pi(\bi)$ be the vector
$$
\Pi(\bi)=\left(\sum_{k=1}^\infty\delta_{i_k}^jc_{\bi|_{k-1}}\right)_{j\in\cA}.
$$
In particular, $\pi_\tau(\bi)=\langle\tau,\Pi(\bi)\rangle$, the scalar product of $\tau=(\tau_i)_{i\in\cA}$ and $\Pi(\bi)$.

It is easy to see that if $\bi\neq\bj\in\Sigma$ then
\begin{equation}\label{eq:trans}
\|\Pi(\bi)-\Pi(\bj)\|\geq |c_{\bi\wedge\bj}|\frac{1-2\max_i|c_i|}{1-\max_i|c_i|}>c |c_{\bi\wedge\bj}|.
\end{equation}

\begin{proof}
	Let $\widehat{\eta_\tau}(\xi)=\int e^{-i\xi\pi_\tau(\bi)}d\nu(\bi)$ be the Fourier transform of $\eta_\tau$. It is enough to verify \cref{eq:enough} for Lebesgue almost every $(\tau_i)_{i\in\cA}$. To show that, it is enough to verify that
	$$
	\iint|\widehat{\eta_\tau}(\xi)|^2|\xi|^td\xi\psi(\tau)d\tau<\infty
	$$
	for every compactly supported density function $\psi\colon\bbR^{\#\cA}\to[0,\infty)$ with Fourier transform $\widehat{\psi}$ satisfying that for every $N\geq1$ there exists a $C_N$ such that for every $\underline{\xi}\in\bbR^{\#\cA}$
	$$
	\widehat{\psi}(\underline{\xi})\leq\frac{C_N}{(1+\|\underline{\xi}\|)^N}.
	$$	
	
	Let us choose $t>1$ and $N>t+1$ such that $\sum_{i\in\cA}|c_i|^{-N}p_i^2<1$. Then
	\begin{align*}
		\left|\iint|\widehat{\eta_\tau}(\xi)|^2|\xi|^td\xi\psi(\tau)d\tau\right|&=\left|\iint\iint e^{i\xi(x-y)}|\xi|^t\psi(\tau)d\eta_\tau(x)d\eta_\tau(y) d\tau d\xi\right|\\
		&=\left|\iint\iint e^{i\xi\langle \tau,\Pi(\bi)-\Pi(\bj)\rangle}\psi(\tau)d\tau|\xi|^td\nu(\bi)d\nu(\bj)d\xi\right|\\
		&=\left|\iiint \widehat{\psi}(\xi\cdot(\Pi(\bi)-\Pi(\bj)))|\xi|^td\nu(\bi)d\nu(\bj)d\xi\right|\\
		&\leq \iiint |\widehat{\psi}(\xi\cdot(\Pi(\bi)-\Pi(\bj)))||\xi|^td\nu(\bi)d\nu(\bj)d\xi\\
		&\leq \iiint \frac{C_N|\xi|^t}{(1+|\xi|\|\Pi(\bi)-\Pi(\bj)\|)^N}d\nu(\bi)d\nu(\bj)d\xi\\
		&\lesssim \iint |c_{\bi\wedge\bj}|^{-N}d\nu(\bi)d\nu(\bj)\int\frac{C_N|\xi|^t}{(1+|\xi|)^N}d\xi\text{ by \cref{eq:trans}}\\
		&\leq \sum_{k=0}^\infty \left(\sum_{i\in\cA}|c_i|^{-N}p_i^2\right)^k\cdot \int\frac{C_N|\xi|^t}{(1+|\xi|)^N}d\xi,
	\end{align*}
which is finite by the choice of $N$ and $t$.
\end{proof}

\begin{proposition}\label{prop:ex1}
	Let $\cA$ be a finite set of indices and for every $i\in\cA$, let $0<|a_i|<|c_i|<1/2$ such that $\sum_{i\in\cA}|c_i||a_i|^{1/4}>1$ and $\sum_{i\in\cA}|a_i|^{1/2}<1$. Let
	\begin{equation}\label{eq:ex1}
		\Phi=\left\{f_i(x)=\begin{pmatrix}
			a_i & 0 \\ 0 & c_i
		\end{pmatrix}x+\begin{pmatrix}t_{i,1}\\t_{i,2}\end{pmatrix}\right\}_{i\in\cA}.
	\end{equation}
	and denote $X$ the attractor of $\Phi$. Then $0<\cH^{s_0}(X)<\infty$ for Lebesgue-almost every $(t_{i})_{i\in\cA}\in\bbR^{2\#\cA}$, where $\sum_{i\in\cA}|c_i||a_i|^{s_0-1}=1$.
\end{proposition}

\begin{proof}
Let $\mu_K$ be the K\"aenm\"aki measure corresponding to the system defined in \cref{eq:ex1}. It is easy to see that for every $\bi\in\Sigma_*$
$$
\mu_K([\bi])=|c_{\bi}||a_{\bi}|^{s_0-1}.
$$
For a proof, see for example \cite{FalconerMiao2007}.

Clearly, $5/4<s_0<3/2$. By the construction, $X_F$ is a singleton containing the direction of the $x$-axis. By the assumption $\sum_{i\in\cA}|a_i|^{1/2}<1$ the result of Rams and V\'ehel \cite[Theorem~1.1]{RamsVehel}, the IFS $\{y\mapsto a_iy+t_{i,1}\}_{i\in\cA}$ satisfies the strong separation condition for Lebesgue almost every $(t_{i,1})_{i\in\cA}$, and so does $\Phi$. On the other hand,
$$
\sum_{i\in\cA}\frac{(|c_i||a_i|^{s_0-1})^2}{|c_i|^2}=\sum_{i\in\cA}|a_i|^{2(s_0-1)}\leq \sum_{i\in\cA}|a_i|^{1/2}<1,
$$
and so, by \cref{lem:fourier1}, the projection of the K\"aenm\"aki measure is absolute continuous with continuous (and thus, bounded) density for Lebesgue almost every $(t_{i,2})_{i\in\cA}$. Then the claim follows by \cref{thm:mainsep}.
\end{proof}

\subsection{Example with positive dimensional Furstenberg directions}

In this section, we consider a dominated example with triangular linear parts for which the Furstenberg measure is supported on a Cantor set.

\begin{proposition}\label{prop:ex2}
	Let
\begin{equation}\label{eq:ex2b}
	\Phi=\left\{f_i(x)=\begin{pmatrix}
		a_i & 0 \\ b_i & c_i
	\end{pmatrix}x+\begin{pmatrix}t_{i,1}\\t_{i,2}\end{pmatrix}\right\}_{i\in\cA}
\end{equation}
be an IFS such that $0<|a_i|<|c_i|<1/2$, $\sum_{i\in\cA}|c_i|>1$, and the linear parts are not simultaneously diagonalisable. Furthermore, suppose that the IFS $\Phi_1=\{x\to a_ix+t_{i,1}\}_{i\in\cA}$ satisfies the strong open set condition. Denote $s_0$ the affinity dimension $\sum_{i\in\cA}|c_i||a_i|^{s_0-1}=1$, $s_0\in(1,2]$. If $\sum_{i\in\cA}|c_i|^{-1}|a_i|^{2(s_0-1)}<1$ then $0<\cH^{s_0}(X)<\infty$ for Lebesgue-almost every $\tau=(t_{i,2})_{i\in\cA}$, where $X$ is the attractor of $\Phi$.
\end{proposition}

Let $\Pi_\tau\colon\bi\mapsto(\pi^1(\bi),\pi^2_\tau(\bi))$ be the natural projection for the IFS $\Phi$. Simple calculation shows that
\begin{equation}\label{eq:coordproj}
\pi^1(\bi)=\sum_{k=1}^\infty t_{i_k,1}a_{\bi|_{k-1}},\text{ and let }\pi_\tau^2(\bi)=\sum_{k=1}^\infty\left(t_{i_k,2}+b_{i_k}\pi^1(\sigma^k\bi)\right)c_{\bi|_{k-1}},
\end{equation}
In particular, $\pi^1\colon\Sigma\to\bbR$ is the natural projection of the IFS $\Phi_1$. Let $\mu_K$ be the K\"aenm\"aki measure, and again by \cite{FalconerMiao2007},
$$
\mu_K([\bi])=|c_{\bi}||a_{\bi}|^{s_0-1}\text{ for every }\bi\in\Sigma_*.
$$
Let us also introduce the natural projection of the IFS $\Phi_2=\{x\to c_ix+t_{i,2}\}_{i\in\cA}$, and let us denote it by
$$
\widetilde{\pi}^2_\tau(\bi)=\sum_{k=1}^\infty t_{i_k,2}c_{\bi|_{k-1}}.
$$
Similarly to the previous case, one can write
$$
\widetilde{\Pi}(\bi)=\left(\sum_{k=1}^\infty\delta_{i_k}^j c_{\bi|_{k-1}}\right)_{j\in\cA},
$$
and $\widetilde{\pi}^2_\tau(\bi)=\langle\tau,\widetilde{\Pi}(\bi)\rangle$. Since $|c_i|<1/2$
\begin{equation}\label{eq:trans2}
	\|\widetilde{\Pi}(\bi)-\widetilde{\Pi}(\bj)\|\geq |c_{\bi\wedge\bj}|\frac{1-2\max_i|c_i|}{1-\max_i|c_i|}>C |c_{\bi\wedge\bj}|.
\end{equation}

With a slight abuse of notation, let $\proj_v(x,y)=y-vx$ for a $v\in\bbR$. So, $\proj_v$ is bi-Lipschitz equivalent to the orthogonal projection to the line $\mathrm{span}\binom{-v}{1}$. It is easy to see that there exists $C>0$ such that the projective interval $$
\cC=\left\{\mathrm{span}\begin{pmatrix}
	v\\ 1
\end{pmatrix}:|v|\leq C\right\}
$$
is invariant with respect to the matrices $A_i^*$. Let $h\colon\bbR\to[0,\infty)$ be a compactly supported continuous density function such that { $\inf_{x\in[-C,C]}h(x)>0$} and for every $M\geq1$ there exists $C_M>0$ such that
\begin{equation}\label{eq:h}
|\widehat{h}(\xi)|\leq\frac{C_M}{(1+|\xi|)^M}\text{ for every $\xi\in\bbR$},
\end{equation}
where $\widehat{h}$ is the Fourier transform of $h$.

\begin{proof}[Proof of \cref{prop:ex2}]
Let us define a compactly supported probability measure $\nu_\tau$ on $\bbR^2$ by
$$
d\nu_\tau(x,y)=h(x)d(\proj_x)_*(\Pi_\tau)_*\mu_K(y)dx.
$$
It is sufficient to show that $\nu_\tau$ is absolutely continuous with continuous density. Indeed, since $h(x)$ is uniformly separated away from zero on $[-C,C]\supseteq X_F$, if $d\nu_\tau(x,y)=g_\tau(x,y)dxdy$ with $g_{\tau}\colon\bbR^2\to[0,\infty)$ continuous, then the measure $(\proj_x)_*(\Pi_\tau)_*\mu_K$ is absolutely continuous with continuous density $g_\tau(x,y)/h(x)$, which is uniformly bounded. This verifies \eqref{it:boundeddens1} of \cref{thm:mainsep}.

By \cref{eq:enough}, it is enough to show for some $t>2$ that
$$
\iiint|\widehat{\nu_\tau}(\xi_1,\xi_2)|^2\|(\xi_1,\xi_2)\|^td\xi_1d\xi_2\psi(\tau)d\tau<\infty
$$
for every compactly supported density function $\psi\colon\bbR^{\#\cA}\to[0,\infty)$ with Fourier transform $\widehat{\psi}$ satisfying that for every $N\geq1$ there exists a $C_N$ such that for every $\underline{\xi}\in\bbR^{\#\cA}$
\begin{equation}\label{eq:psi}
\widehat{\psi}(\underline{\xi})\leq\frac{C_N}{(1+\|\underline{\xi}\|)^N}.
\end{equation}

By definition
$$
\proj_x(\Pi_\tau(\bi))=\pi_\tau^2(\bi)-x\pi^1(\bi)=\langle\tau,\widetilde{\Pi}(\bi)\rangle-x\pi^1(\bi)+\sum_{k=1}^\infty b_{i_k}\pi^1(\sigma^k\bi)c_{\bi|_{k-1}}.
$$
Let us choose $t>2$ and $N,M>t+1$ such that $\sum_{i\in\cA}|c_i|^{2-N}|a_i|^{2(s_0-1)}<1$. Simple algebraic manipulations show that
\begin{align*}
	&\left|\iiint|\widehat{\nu_\tau}(\xi_1,\xi_2)|^2\|(\xi_1,\xi_2)\|^td\xi_1d\xi_2\psi(\tau)d\tau\right|\\
	&=\left|\iint\|(\xi_1,\xi_2)\|^t\int\iint\iint e^{i\xi_1(x-y)+i\xi_2\left(\proj_x(\Pi_\tau(\bi))-\proj_y(\Pi_\tau(\bj))\right)}h(x)h(y)\psi(\tau)d\mu_K(\bi)d\mu_K(\bj)dxdyd\tau d\xi_1d\xi_2\right|\\
	&=\left|\iint\|(\xi_1,\xi_2)\|^t\int\iint\iint e^{ix\left(\xi_1-\xi_2\pi^1(\bi)\right)+y\left(\xi_2\pi^1(\bj)-\xi_1\right)+i\xi_2\left(\pi_\tau^2(\bi)-\pi_\tau^2(\bj)\right)}h(x)h(y)\psi(\tau)dxdyd\mu_K(\bi)d\mu_K(\bj)d\tau d\xi_1d\xi_2\right|\\
	&=\left|\iint\|(\xi_1,\xi_2)\|^t\int\iint\widehat{h}\left(\xi_1-\xi_2\pi^1(\bi)\right)\widehat{h}\left(\xi_2\pi^1(\bj)-\xi_1\right) e^{i\xi_2\left(\pi_\tau^2(\bi)-\pi_\tau^2(\bj)\right)}\psi(\tau)d\mu_K(\bi)d\mu_K(\bj)d\tau d\xi_1d\xi_2\right|\\
	&\leq\iint\|(\xi_1,\xi_2)\|^t\iint\left|\widehat{h}\left(\xi_1-\xi_2\pi^1(\bi)\right)\right|\left|\widehat{h}\left(\xi_2\pi^1(\bj)-\xi_1\right)\right|\left|\int e^{i\xi_2\left(\pi_\tau^2(\bi)-\pi_\tau^2(\bj)\right)}\psi(\tau)d\tau\right|d\mu_K(\bi)d\mu_K(\bj)d\xi_1d\xi_2\\
	&=\iint\|(\xi_1,\xi_2)\|^t\iint\left|\widehat{h}\left(\xi_1-\xi_2\pi^1(\bi)\right)\right|\left|\widehat{h}\left(\xi_2\pi^1(\bj)-\xi_1\right)\right|\left|\int e^{i\xi_2\langle\tau,\widetilde{\Pi}(\bi)-\widetilde{\Pi}(\bj)\rangle}\psi(\tau)d\tau\right|d\mu_K(\bi)d\mu_K(\bj)d\xi_1d\xi_2\\
	&=\iint\|(\xi_1,\xi_2)\|^t\iint\left|\widehat{h}\left(\xi_1-\xi_2\pi^1(\bi)\right)\right|\left|\widehat{h}\left(\xi_2\pi^1(\bj)-\xi_1\right)\right|\left|\widehat{\psi}\left(\xi_2(\widetilde{\Pi}(\bi)-\widetilde{\Pi}(\bj)\right)\right|d\mu_K(\bi)d\mu_K(\bj)d\xi_1d\xi_2\\
	\intertext{by using \cref{eq:psi} and \cref{eq:h}}
	&\leq\iint\|(\xi_1,\xi_2)\|^t\iint\frac{C_N\left|\widehat{h}\left(\xi_1-\xi_2\pi^1(\bi)\right)\right|\left|\widehat{h}\left(\xi_2\pi^1(\bj)-\xi_1\right)\right|}{(1+|\xi_2|\|\widetilde{\Pi}(\bi)-\widetilde{\Pi}(\bj)\|)^N}d\mu_K(\bi)d\mu_K(\bj)d\xi_1d\xi_2\\
	&\leq\iint\iint\frac{C_NC_M\|(\xi_1,\xi_2)\|^t}{(1+|\xi_1-\xi_2\pi^1(\bi)|)^M(1+|\xi_2|\|\widetilde{\Pi}(\bi)-\widetilde{\Pi}(\bj)\|)^N}d\xi_1d\xi_2d\mu_K(\bi)d\mu_K(\bj)\\
	&\leq\iint\iint\frac{C_NC_M\|(\xi_1,\xi_2)\|^t}{(1+|\xi_1-\xi_2\pi^1(\bi)|)^M(1+|\xi_2|)^N\|\widetilde{\Pi}(\bi)-\widetilde{\Pi}(\bj)\|^N}d\xi_1d\xi_2d\mu_K(\bi)d\mu_K(\bj)\\
	\intertext{by using the coordinate change $\xi_1'=\xi_1-\pi^1(\bi)\xi_2$ and $\xi_2'=\xi_2$, observe that $\|(\xi_1'+\pi^1(\bi)\xi_2',\xi_2')\|^t\leq\left(\|(\xi_1',\xi_2')\|+|\pi^1(\bi)||\xi_2'|\right)^t\leq2^t\|(\xi_1',\xi_2')\|^t$, and so}
	&=C_NC_M\iint\frac{\|(\xi_1',\xi_2')\|^t}{(1+|\xi_1'|)^M(1+|\xi_2'|)^N}d\xi_1'd\xi_2'\iint\|\widetilde{\Pi}(\bi)-\widetilde{\Pi}(\bj)\|^{-N}d\mu_K(\bi)d\mu_K(\bj)\\
	&\lesssim\iint\frac{\|(\xi_1',\xi_2')\|^t}{(1+|\xi_1'|)^M(1+|\xi_2'|)^N}d\xi_1'd\xi_2'\sum_{k=1}^\infty \left(\sum_{i\in\cA}|c_i|^{2-N}|a_i|^{2(s_0-1)}\right)^k,
\end{align*}
where in the last step, we applied \cref{eq:trans2}. Now, the right-hand side is finite by the choice of parameters, $t,N$ and $M$.
\end{proof}

\section*{Acknowledgement} The author is grateful to Antti K\"aenm\"aki for the useful discussions and his valuable comments. The author further thanks the anonymous referee for their careful reading and valuable suggestions.

\bibliographystyle{abbrv}
\bibliography{Bibliography}

\end{document}